\documentclass[10pt]{article}

\usepackage{tikz, float}
\usepackage{amssymb}
\usepackage{url}
\usepackage{amscd}
\usepackage{amsthm}
\usepackage{amsmath}
\usepackage{graphicx}

\newtheorem {Theorem}  {Theorem}

\newtheorem {Conjecture} {Conjecture}

\newtheorem {Problem} {Problem}

\begin{document}
\baselineskip = 15pt
\bibliographystyle{plain}

\title{Undecidability of Translational Tiling of the $4$-dimensional Space\\ with a Set of $4$ Polyhypercubes}
\date{}
\author{Chao Yang\\ 
              School of Mathematics and Statistics\\
              Guangdong University of Foreign Studies, Guangzhou, 510006, China\\
              sokoban2007@163.com, yangchao@gdufs.edu.cn\\
              \\
        Zhujun Zhang\\
              Big Data Center of Fengxian District, Shanghai, 201499, China\\
              zhangzhujun1988@163.com
              }
\maketitle

\begin{abstract}
Recently, Greenfeld and Tao disprove the conjecture that translational tilings of a single tile can always be periodic [Ann. Math. 200(2024), 301-363]. In another paper [to appear in J. Eur. Math. Soc.], they also show that if the dimension $n$ is part of the input, the translational tiling for subsets of $\mathbb{Z}^n$ with one tile is undecidable. These two results are very strong pieces of evidence for the conjecture that translational tiling of $\mathbb{Z}^n$ with a monotile is undecidable, for some fixed $n$. This paper shows that translational tiling of the $3$-dimensional space with a set of $5$ polycubes is undecidable. By introducing a technique that lifts a set of polycubes and its tiling from $3$-dimensional space to $4$-dimensional space, we manage to show that translational tiling of the $4$-dimensional space with a set of $4$ tiles is undecidable. This is a step towards the attempt to settle the conjecture of the undecidability of translational tiling of the $n$-dimensional space with a monotile, for some fixed $n$.
\end{abstract}


\section{Introduction} 

The decidability or undecidability of the following translational tiling problem has received extensive studies recently \cite{gt21,gt23,gt24a,gt24b,yang23,yang23b,yz24,yz24b}. 

\begin{Problem}[Translational tiling of $\mathbb{Z}^n$ with a set of $k$ tiles] \label{pro_main}
A tile is a finite subset of $\mathbb{Z}^n$. Let $k$ and $n$ be fixed positive integers. Given a set $S$ of $k$ tiles in $\mathbb{Z}^n$, is there an algorithm to decide whether $\mathbb{Z}^n$ can be tiled by translated copies of tiles in $S$?
\end{Problem}

For dimension one ($n=1$), it is shown that for any $k$, if a set of $k$ tiles can tile $\mathbb{Z}$, then it can always tile $\mathbb{Z}$ periodically \cite{s93}, even though the periods can be exponentially long \cite{kolo03, st05}. As a consequence, the problem is decidable in dimension one. For $n=2$ and $k=1$, the periodicity and decidability are shown in \cite{bn91,b20, w15}. Structures of translational tilings of $\mathbb{Z}^n$ are studied in \cite{gt21}. Researchers have long believed the translational tiling of $\mathbb{Z}^n$ must also be decidable for one tile (i.e. $k$=1) by posting the following conjecture.

\begin{Conjecture}[Periodic Tiling Conjecture, \cite{gs16,lw96,s74}]
Let $T$ be a tile in $\mathbb{Z}^n$. If $T$ tiles $\mathbb{Z}^n$ with translated copies, then it can tile $\mathbb{Z}^n$ periodically with translated copies.
\end{Conjecture}

The periodic tiling conjecture implies the decidability of translational tiling with one tile. However, the periodic tiling conjecture is disproved by Greenfeld and Tao \cite{gt24a}, which suggests that the translational tiling problem might be undecidable even with just one tile. Greenfeld and Tao also show that if the dimension $n$ is part of the input, the translational tiling for subsets of $\mathbb{Z}^n$ with one tile is undecidable \cite{gt24b}. The undecidability of the translational tiling problem of the entire space $\mathbb{Z}^n$ for some fixed $n$ with one tile ($k=1$) is still open.

The first undecidability result on tiling is Berger's proof of undecidability of Wang's domino problem.  A \textit{Wang tile} is a unit square with each edge assigned a color. Given a finite set of Wang tiles (see Figure \ref{fig_w3} for an example), Wang considered the problem of tiling the entire plane with translated copies from the set, under the conditions that the tiles must be edge-to-edge and the color of common edges of any two adjacent Wang tiles must be the same \cite{wang61}. This is known as \textit{Wang's domino problem}. Berger showed that Wang's domino problem is undecidable in general (i.e. the size of the set of Wang tiles can be arbitrarily large) in the 1960s.

\begin{Theorem}[\cite{b66}]\label{thm_berger}
    Wang's domino problem is undecidable.
\end{Theorem}


\begin{figure}[ht]
\begin{center}
\begin{tikzpicture}[scale=0.5]

\draw [ fill=green!50] (0,0)--(2,2)--(2,0)--(0,0);
\draw [ fill=red!50] (0,0)--(2,-2)--(2,0)--(0,0);
\draw [ fill=red!50] (4,0)--(2,2)--(2,0)--(4,0);
\draw [ fill=yellow!50] (4,0)--(2,-2)--(2,0)--(4,0);

\draw [ fill=yellow!50] (6+0,0)--(6+2,2)--(6+2,0)--(6+0,0);
\draw [ fill=blue!50] (6+0,0)--(6+2,-2)--(6+2,0)--(6+0,0);
\draw [ fill=blue!50] (6+4,0)--(6+2,2)--(6+2,0)--(6+4,0);
\draw [ fill=red!50] (6+4,0)--(6+2,-2)--(6+2,0)--(6+4,0);

\draw [ fill=red!50] (12+0,0)--(12+2,2)--(12+2,0)--(12+0,0);
\draw [ fill=yellow!50] (12+0,0)--(12+2,-2)--(12+2,0)--(12+0,0);
\draw [ fill=yellow!50] (12+4,0)--(12+2,2)--(12+2,0)--(12+4,0);
\draw [ fill=green!50] (12+4,0)--(12+2,-2)--(12+2,0)--(12+4,0);

\end{tikzpicture}
\end{center}
\caption{A set of $3$ Wang tiles} \label{fig_w3}
\end{figure}
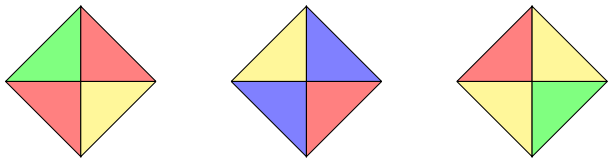

By reduction from Wang's domino problem, Ollinger obtained the first undecidability result with a fixed size of tile sets \cite{o09}. Ollinger showed that it is undecidable to tile the plane with a set of $11$ polyominoes. Ollinger's result was subsequently improved by Yang and Zhang \cite{yang23,yang23b,yz24}. In a series of works, Yang and Zhang proved that $8$ polyominoes suffice to deduce the undecidability of translational tiling of the plane, which confirms a conjecture of Ollinger.

By introducing a technique for lifting a $2$-dimensional tiling to $3$-dimensional space, Yang and Zhang show that it is undecidable to tile the $3$-dimensional space with $6$ polycubes \cite{yz24b}. Roughly speaking, the lifting technique decreases the number of tiles with a trade-off of increasing the dimension of the space by one.

\begin{Theorem}[\cite{yz24b}]\label{thm_3d_old}
    It is undecidable to tile the $3$-dimensional space with translated copies of a set of $6$ polycubes.
\end{Theorem}

The main contribution of this paper is to prove the following undecidability result with $4$ tiles in $4$-dimensional space, which is a step towards the undecidability of translational tiling of $\mathbb{Z}^n$ for some sufficiently large $n$ with one tile.

\begin{Theorem}[Undecidability with Four Tiles]\label{thm_main}
    It is undecidable to tile the $4$-dimensional space with translated copies of a set of $4$ polyhypercubes.
\end{Theorem}

To prove Theorem \ref{thm_main}, we first prove Theorem \ref{thm_3d_new} on the undecidability of translational tiling of $3$-dimensional space with $5$ tiles, which improves the $6$-tile undecidability result (Theorem \ref{thm_3d_old}). A novel framework of reduction is introduced in the proof of Theorem \ref{thm_3d_new}, without using the lifting technique. Note that it is crucial that the new proof method is not a direct lift from a $2$-dimensional tiling, as the lifting technique cannot be applied twice. The main result, Theorem \ref{thm_main}, is then proved by lifting the $3$-dimensional tiling of Theorem \ref{thm_3d_new} to $4$-dimensional space.

\begin{Theorem}[Undecidability with Five Tiles]\label{thm_3d_new}
    It is undecidable to tile the $3$-dimensional space with translated copies of a set of $5$ polycubes.
\end{Theorem}

The current state of knowledge about the decidability and undecidability of the translational tiling problem of $\mathbb{Z}^n$ with a set of $k$ tiles for the fixed pairs $(k,n)$ can be summarized in Figure \ref{fig_nk}. The green region is known to be decidable, and the red region is (possibly) undecidable. Note that the frontier of the region of undecidable is not clearly known yet, especially for dimensions $5$ and above. The dashed line of the boundary of the undecidable region in Figure \ref{fig_nk} is to demonstrate the idea that as the dimension goes up, it may require fewer tiles to get undecidable results for translational tiling of $\mathbb{Z}^n$.

\begin{figure}[H]
\begin{center}
\begin{tikzpicture}

\draw [fill=green!20,dashed] (11.9,0)--(0,0)--(0,2)--(1,2)--(1,1)--(11.9,1);

\begin{scope}
    \clip (0,9) rectangle (2,10.9);
\draw [fill=red!20,dashed] (0,11)--(0,9)--(3,9)--(3,11); 
\end{scope}

\begin{scope}
    \clip (3,1) rectangle (11.9,4);
\draw [fill=red!20,dashed] (12,1)--(7,1)--(7,2)--(4,2)--(4,3)--(3,3)--(3,4)--(12,4); 
\end{scope}

\begin{scope}
    \clip (1.64,3) rectangle (11.9,10.9);

\draw [dashed,fill=red!20] plot [smooth] coordinates {(0,11) (2,8.5) (2.5,6.5) (3,4) (3.5,1) (10,0) (13,7) (13,15) (5,12)};
\end{scope}

\draw[help lines, color=gray, dashed]  (-0.1,-0.1) grid (11.9,10.9);
\draw[->,ultra thick] (-0.5,0)--(12,0) node[right]{$k$};
\draw[->,ultra thick] (0,-0.5)--(0,11) node[above]{$n$};

\foreach \y in {1,2,3,4,5}
{
\node at (-0.5,\y-0.5) {\y};
}
\node at (-0.9,9.7) {some};
\node at (-0.8,9.3) {large $n$?};

\foreach \x in {1,...,11}
{
\node at (\x-0.5,-0.5) {\x};
}

\end{tikzpicture}
\end{center}
\caption{Translational tiling problem of $\mathbb{Z}^n$ with a set of $k$ tiles.}\label{fig_nk}
\end{figure}
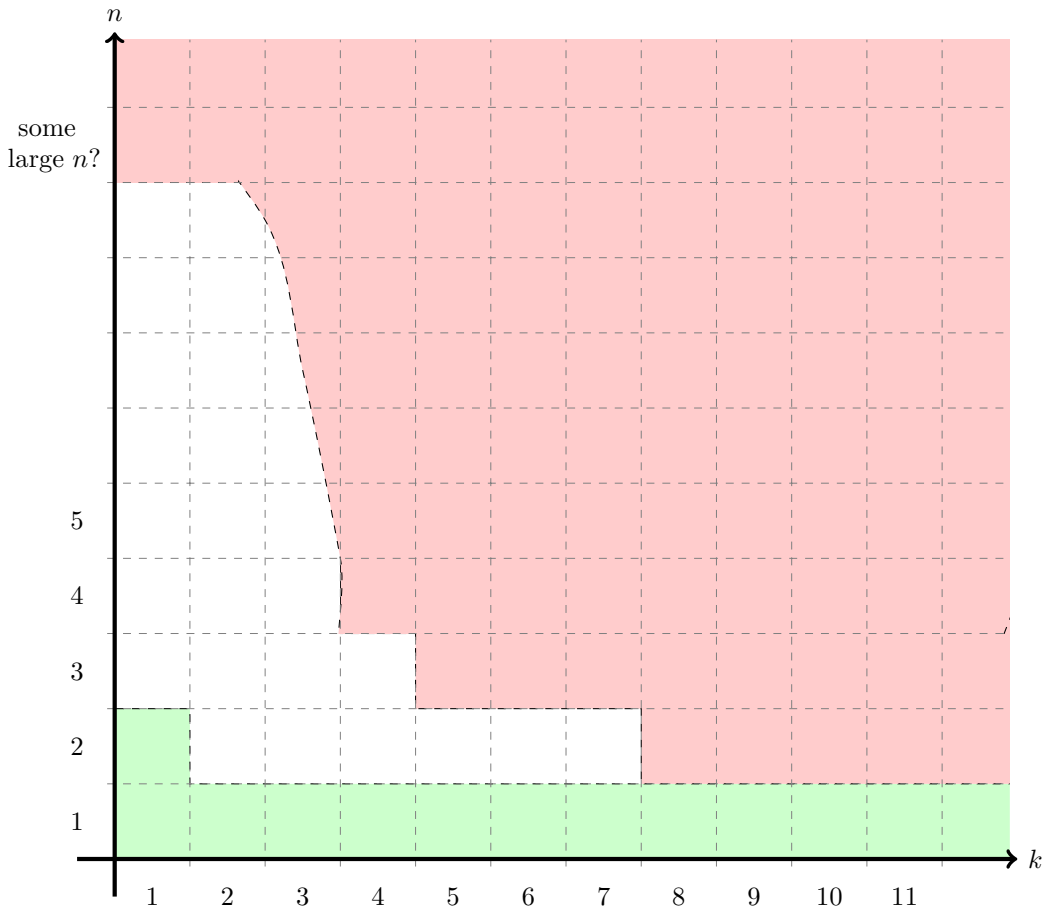

The rest of the paper is organized as follows. Section \ref{sec_3d} proves Theorem \ref{thm_3d_new} by introducing a reduction framework different from that of Ollinger in several aspects. Section \ref{sec_4d} proves Theorem \ref{thm_main} by lifting the $3$-dimensional tile set and its tiling in the proof of Theorem \ref{thm_3d_new} to $4$-dimensional space. Section \ref{sec_conclu} concludes with a few remarks on future work.

\section{Undecidability of Tiling $3$-dimensional Space}\label{sec_3d}

In this section, we will prove Theorem \ref{thm_3d_new} by constructing a set of $5$ polycubes which is compatible with the lifting technique (i.e. the number of tiles can be decreased by applying the technique). 

\subsection{Building Blocks}

We first define several building blocks which will be used in the construction of the more sophisticated polycubes in the $5$-tile set in next subsection. A \textit{polycube} is the union of a finit number of $1\times 1\times 1$ unit cubes gluing together face-to-face. A \textit{functional cube} is a $8\times 8\times 8$ polycube. Figure \ref{fig_3d_c} is the layer diagram for building block $c$ for encoding the colors of Wang tiles, from the bottom most layer (first layer) to the top most layer (eighth layer). The building block $c$ is a functional cube with a dent on its north side. Figure \ref{fig_3d_c_minus} is the layer diagram for building block $c^-$, which is a $8\times 7\times 8$ polycube with a dent on its north side. The shapes of the dents of the building blocks $c$ and $c^-$ are exactly the same. The only difference of these two building blocks is that the base polycube of $c^-$ is one unit shorter in the south-north direction than that of $c$.

Note that all figures in this subsection is depicted in \textit{level-1} layer diagrams, where each gray square represents a $1\times 1\times 1$ unit cube.


\begin{figure}[H]
\begin{center}
\begin{tikzpicture}[scale=0.4]

\foreach \x in {0,10,20,30}
\foreach \y in {0,...,8} 
{ 
\draw  (\x+0,0+\y)--(\x+8,0+\y);
\draw  (\x,0)--(\x,8);
\draw  (\x+8,0)--(\x+8,8);
\draw  (\x+7,0)--(\x+7,8);
\draw  (\x+6,0)--(\x+6,8);
\draw  (\x+5,0)--(\x+5,8);
\draw  (\x+4,0)--(\x+4,8);
\draw  (\x+3,0)--(\x+3,8);
\draw  (\x+2,0)--(\x+2,8);
\draw  (\x+1,0)--(\x+1,8);
}

\foreach \x in {0,...,7}
\foreach \y in {0,...,7} 
{
\draw [fill=gray!20] (\x+0,1+\y)--(\x+1,1+\y)--(\x+1,0+\y)--(\x+0,0+\y)--(\x+0,1+\y);
}

\foreach \x in {10,...,17}
\foreach \y in {0,1,2,3,4,7} 
{
\draw [fill=gray!20] (\x+0,1+\y)--(\x+1,1+\y)--(\x+1,0+\y)--(\x+0,0+\y)--(\x+0,1+\y);
}
\foreach \x in {10,17}
\foreach \y in {5,6} 
{
\draw [fill=gray!20] (\x+0,1+\y)--(\x+1,1+\y)--(\x+1,0+\y)--(\x+0,0+\y)--(\x+0,1+\y);
}

\foreach \x in {20,...,27}
\foreach \y in {0,1,2,3,4,7} 
{
\draw [fill=gray!20] (\x+0,1+\y)--(\x+1,1+\y)--(\x+1,0+\y)--(\x+0,0+\y)--(\x+0,1+\y);
}
\foreach \x in {20,27}
\foreach \y in {5,6} 
{
\draw [fill=gray!20] (\x+0,1+\y)--(\x+1,1+\y)--(\x+1,0+\y)--(\x+0,0+\y)--(\x+0,1+\y);
}
\foreach \x in {22,...,25}
\foreach \y in {6} 
{
\draw [fill=gray!20] (\x+0,1+\y)--(\x+1,1+\y)--(\x+1,0+\y)--(\x+0,0+\y)--(\x+0,1+\y);
}

\foreach \x in {30,...,37}
\foreach \y in {0,...,4} 
{
\draw [fill=gray!20] (\x+0,1+\y)--(\x+1,1+\y)--(\x+1,0+\y)--(\x+0,0+\y)--(\x+0,1+\y);
}
\foreach \x in {30,37}
\foreach \y in {5,6} 
{
\draw [fill=gray!20] (\x+0,1+\y)--(\x+1,1+\y)--(\x+1,0+\y)--(\x+0,0+\y)--(\x+0,1+\y);
}
\foreach \x in {32,35}
\foreach \y in {6} 
{
\draw [fill=gray!20] (\x+0,1+\y)--(\x+1,1+\y)--(\x+1,0+\y)--(\x+0,0+\y)--(\x+0,1+\y);
}
\foreach \x in {30,31,32,35,36,37}
\foreach \y in {7} 
{
\draw [fill=gray!20] (\x+0,1+\y)--(\x+1,1+\y)--(\x+1,0+\y)--(\x+0,0+\y)--(\x+0,1+\y);
}

\node at (4,-1) {$1$st layer};  \node at (14,-1) {$2$nd layer};  \node at (24,-1) {$3$rd layer};  \node at (34,-1) {$4$th and 5th layers};

\node at (4,-11) {$6$th layer};  \node at (14,-11) {$7$th layer};  \node at (24,-11) {$8$th layer};


\foreach \x in {0,10,20}
\foreach \y in {-10,...,-2} 
{ 
\draw  (\x+0,0+\y)--(\x+8,0+\y);
\draw  (\x,-10)--(\x,-2);
\draw  (\x+8,-10)--(\x+8,-2);
\draw  (\x+7,-10)--(\x+7,-2);
\draw  (\x+6,-10)--(\x+6,-2);
\draw  (\x+5,-10)--(\x+5,-2);
\draw  (\x+4,-10)--(\x+4,-2);
\draw  (\x+3,-10)--(\x+3,-2);
\draw  (\x+2,-10)--(\x+2,-2);
\draw  (\x+1,-10)--(\x+1,-2);
}

\foreach \x in {0,...,7}
\foreach \y in {-10,-9,-8,-7,-6,-3} 
{
\draw [fill=gray!20] (\x+0,1+\y)--(\x+1,1+\y)--(\x+1,0+\y)--(\x+0,0+\y)--(\x+0,1+\y);
}
\foreach \x in {0,7}
\foreach \y in {-5,-4} 
{
\draw [fill=gray!20] (\x+0,1+\y)--(\x+1,1+\y)--(\x+1,0+\y)--(\x+0,0+\y)--(\x+0,1+\y);
}
\foreach \x in {2,...,5}
\foreach \y in {-4} 
{
\draw [fill=gray!20] (\x+0,1+\y)--(\x+1,1+\y)--(\x+1,0+\y)--(\x+0,0+\y)--(\x+0,1+\y);
}

\foreach \x in {10,...,17}
\foreach \y in {-10,-9,-8,-7,-6,-3} 
{
\draw [fill=gray!20] (\x+0,1+\y)--(\x+1,1+\y)--(\x+1,0+\y)--(\x+0,0+\y)--(\x+0,1+\y);
}
\foreach \x in {10,17}
\foreach \y in {-5,-4} 
{
\draw [fill=gray!20] (\x+0,1+\y)--(\x+1,1+\y)--(\x+1,0+\y)--(\x+0,0+\y)--(\x+0,1+\y);
}

\foreach \x in {20,...,27}
\foreach \y in {-10,...,-3} 
{
\draw [fill=gray!20] (\x+0,1+\y)--(\x+1,1+\y)--(\x+1,0+\y)--(\x+0,0+\y)--(\x+0,1+\y);
}

\end{tikzpicture}
\end{center}
\caption{Level-1 layer diagram of $c$, a functional cube with a dent.}\label{fig_3d_c}
\end{figure}


\begin{figure}[H]
\begin{center}
\begin{tikzpicture}[scale=0.4]

\foreach \x in {0,10,20,30}
\foreach \y in {0,...,8} 
{ 
\draw  (\x+0,0+\y)--(\x+8,0+\y);
\draw  (\x,0)--(\x,8);
\draw  (\x+8,0)--(\x+8,8);
\draw  (\x+7,0)--(\x+7,8);
\draw  (\x+6,0)--(\x+6,8);
\draw  (\x+5,0)--(\x+5,8);
\draw  (\x+4,0)--(\x+4,8);
\draw  (\x+3,0)--(\x+3,8);
\draw  (\x+2,0)--(\x+2,8);
\draw  (\x+1,0)--(\x+1,8);
}

\foreach \x in {0,...,7}
\foreach \y in {0,...,6} 
{
\draw [fill=gray!20] (\x+0,1+\y)--(\x+1,1+\y)--(\x+1,0+\y)--(\x+0,0+\y)--(\x+0,1+\y);
}

\foreach \x in {10,...,17}
\foreach \y in {0,1,2,3,6} 
{
\draw [fill=gray!20] (\x+0,1+\y)--(\x+1,1+\y)--(\x+1,0+\y)--(\x+0,0+\y)--(\x+0,1+\y);
}
\foreach \x in {10,17}
\foreach \y in {4,5} 
{
\draw [fill=gray!20] (\x+0,1+\y)--(\x+1,1+\y)--(\x+1,0+\y)--(\x+0,0+\y)--(\x+0,1+\y);
}

\foreach \x in {20,...,27}
\foreach \y in {0,1,2,3,6} 
{
\draw [fill=gray!20] (\x+0,1+\y)--(\x+1,1+\y)--(\x+1,0+\y)--(\x+0,0+\y)--(\x+0,1+\y);
}
\foreach \x in {20,27}
\foreach \y in {4,5} 
{
\draw [fill=gray!20] (\x+0,1+\y)--(\x+1,1+\y)--(\x+1,0+\y)--(\x+0,0+\y)--(\x+0,1+\y);
}
\foreach \x in {22,...,25}
\foreach \y in {5} 
{
\draw [fill=gray!20] (\x+0,1+\y)--(\x+1,1+\y)--(\x+1,0+\y)--(\x+0,0+\y)--(\x+0,1+\y);
}

\foreach \x in {30,...,37}
\foreach \y in {0,...,3} 
{
\draw [fill=gray!20] (\x+0,1+\y)--(\x+1,1+\y)--(\x+1,0+\y)--(\x+0,0+\y)--(\x+0,1+\y);
}
\foreach \x in {30,37}
\foreach \y in {5,4} 
{
\draw [fill=gray!20] (\x+0,1+\y)--(\x+1,1+\y)--(\x+1,0+\y)--(\x+0,0+\y)--(\x+0,1+\y);
}
\foreach \x in {32,35}
\foreach \y in {5} 
{
\draw [fill=gray!20] (\x+0,1+\y)--(\x+1,1+\y)--(\x+1,0+\y)--(\x+0,0+\y)--(\x+0,1+\y);
}
\foreach \x in {30,31,32,35,36,37}
\foreach \y in {6} 
{
\draw [fill=gray!20] (\x+0,1+\y)--(\x+1,1+\y)--(\x+1,0+\y)--(\x+0,0+\y)--(\x+0,1+\y);
}

\node at (4,-1) {$1$st layer};  \node at (14,-1) {$2$nd layer};  \node at (24,-1) {$3$rd layer};  \node at (34,-1) {$4$th and 5th layers};

\node at (4,-11) {$6$th layer};  \node at (14,-11) {$7$th layer};  \node at (24,-11) {$8$th layer};


\foreach \x in {0,10,20}
\foreach \y in {-10,...,-2} 
{ 
\draw  (\x+0,0+\y)--(\x+8,0+\y);
\draw  (\x,-10)--(\x,-2);
\draw  (\x+8,-10)--(\x+8,-2);
\draw  (\x+7,-10)--(\x+7,-2);
\draw  (\x+6,-10)--(\x+6,-2);
\draw  (\x+5,-10)--(\x+5,-2);
\draw  (\x+4,-10)--(\x+4,-2);
\draw  (\x+3,-10)--(\x+3,-2);
\draw  (\x+2,-10)--(\x+2,-2);
\draw  (\x+1,-10)--(\x+1,-2);
}

\foreach \x in {0,...,7}
\foreach \y in {-10,-9,-8,-7,-4} 
{
\draw [fill=gray!20] (\x+0,1+\y)--(\x+1,1+\y)--(\x+1,0+\y)--(\x+0,0+\y)--(\x+0,1+\y);
}
\foreach \x in {0,7}
\foreach \y in {-6,-5} 
{
\draw [fill=gray!20] (\x+0,1+\y)--(\x+1,1+\y)--(\x+1,0+\y)--(\x+0,0+\y)--(\x+0,1+\y);
}
\foreach \x in {2,...,5}
\foreach \y in {-5} 
{
\draw [fill=gray!20] (\x+0,1+\y)--(\x+1,1+\y)--(\x+1,0+\y)--(\x+0,0+\y)--(\x+0,1+\y);
}

\foreach \x in {10,...,17}
\foreach \y in {-10,-9,-8,-7,-4} 
{
\draw [fill=gray!20] (\x+0,1+\y)--(\x+1,1+\y)--(\x+1,0+\y)--(\x+0,0+\y)--(\x+0,1+\y);
}
\foreach \x in {10,17}
\foreach \y in {-5,-6} 
{
\draw [fill=gray!20] (\x+0,1+\y)--(\x+1,1+\y)--(\x+1,0+\y)--(\x+0,0+\y)--(\x+0,1+\y);
}

\foreach \x in {20,...,27}
\foreach \y in {-10,...,-4} 
{
\draw [fill=gray!20] (\x+0,1+\y)--(\x+1,1+\y)--(\x+1,0+\y)--(\x+0,0+\y)--(\x+0,1+\y);
}

\end{tikzpicture}
\end{center}
\caption{Level-1 layer diagram of $c^-$, a functional cube with a dent.}\label{fig_3d_c_minus}
\end{figure}

The building blocks $C$ (see Figure \ref{fig_3d_C}) is a functional cube with a bump on its south side. The shape of the bump of $C$ matches exactly the shape of the dent of $c$. In fact, a building block $C$ can be placed properly to the north of a building block $c$ to form a $8\times 16\times 8$ cuboid polycube without gaps or overlaps.


\begin{figure}[H]
\begin{center}
\begin{tikzpicture}[scale=0.4]

\foreach \x in {0,10,20,30}
\foreach \y in {0,...,8} 
{ 
\draw  (\x+0,0+\y)--(\x+8,0+\y);
\draw  (\x,0)--(\x,8);
\draw  (\x+8,0)--(\x+8,8);
\draw  (\x+7,0)--(\x+7,8);
\draw  (\x+6,0)--(\x+6,8);
\draw  (\x+5,0)--(\x+5,8);
\draw  (\x+4,0)--(\x+4,8);
\draw  (\x+3,0)--(\x+3,8);
\draw  (\x+2,0)--(\x+2,8);
\draw  (\x+1,0)--(\x+1,8);
}

\foreach \x in {0,...,7}
\foreach \y in {0,...,7} 
{
\draw [fill=gray!20] (\x+0,1+\y)--(\x+1,1+\y)--(\x+1,0+\y)--(\x+0,0+\y)--(\x+0,1+\y);
}

\foreach \x in {10,...,17}
\foreach \y in {0,...,7} 
{
\draw [fill=gray!20] (\x+0,1+\y)--(\x+1,1+\y)--(\x+1,0+\y)--(\x+0,0+\y)--(\x+0,1+\y);
}
\foreach \x in {11,...,16}
\foreach \y in {-2,-3} 
{
\draw [fill=gray!20] (\x+0,1+\y)--(\x+1,1+\y)--(\x+1,0+\y)--(\x+0,0+\y)--(\x+0,1+\y);
}

\foreach \x in {20,...,27}
\foreach \y in {0,...,7} 
{
\draw [fill=gray!20] (\x+0,1+\y)--(\x+1,1+\y)--(\x+1,0+\y)--(\x+0,0+\y)--(\x+0,1+\y);
}
\foreach \x in {21,26}
\foreach \y in {-2} 
{
\draw [fill=gray!20] (\x+0,1+\y)--(\x+1,1+\y)--(\x+1,0+\y)--(\x+0,0+\y)--(\x+0,1+\y);
}
\foreach \x in {21,...,26}
\foreach \y in {-3} 
{
\draw [fill=gray!20] (\x+0,1+\y)--(\x+1,1+\y)--(\x+1,0+\y)--(\x+0,0+\y)--(\x+0,1+\y);
}

\foreach \x in {30,...,37}
\foreach \y in {0,...,7} 
{
\draw [fill=gray!20] (\x+0,1+\y)--(\x+1,1+\y)--(\x+1,0+\y)--(\x+0,0+\y)--(\x+0,1+\y);
}
\foreach \x in {33,34}
\foreach \y in {-1,-2} 
{
\draw [fill=gray!20] (\x+0,1+\y)--(\x+1,1+\y)--(\x+1,0+\y)--(\x+0,0+\y)--(\x+0,1+\y);
}
\foreach \x in {31,36}
\foreach \y in {-2} 
{
\draw [fill=gray!20] (\x+0,1+\y)--(\x+1,1+\y)--(\x+1,0+\y)--(\x+0,0+\y)--(\x+0,1+\y);
}
\foreach \x in {31,...,36}
\foreach \y in {-3} 
{
\draw [fill=gray!20] (\x+0,1+\y)--(\x+1,1+\y)--(\x+1,0+\y)--(\x+0,0+\y)--(\x+0,1+\y);
}

\node at (4,-4) {$1$st layer};  \node at (14,-4) {$2$nd layer};  \node at (24,-4) {$3$rd layer};  \node at (34,-4) {$4$th and 5th layers};

\node at (4,-17) {$6$th layer};  \node at (14,-17) {$7$th layer};  \node at (24,-17) {$8$th layer};


\foreach \x in {0,10,20}
\foreach \y in {-13,...,-5} 
{ 
\draw  (\x+0,0+\y)--(\x+8,0+\y);
\draw  (\x,-13)--(\x,-5);
\draw  (\x+8,-13)--(\x+8,-5);
\draw  (\x+7,-13)--(\x+7,-5);
\draw  (\x+6,-13)--(\x+6,-5);
\draw  (\x+5,-13)--(\x+5,-5);
\draw  (\x+4,-13)--(\x+4,-5);
\draw  (\x+3,-13)--(\x+3,-5);
\draw  (\x+2,-13)--(\x+2,-5);
\draw  (\x+1,-13)--(\x+1,-5);
}

\foreach \x in {0,...,7}
\foreach \y in {-13,...,-6} 
{
\draw [fill=gray!20] (\x+0,1+\y)--(\x+1,1+\y)--(\x+1,0+\y)--(\x+0,0+\y)--(\x+0,1+\y);
}
\foreach \x in {1,...,6}
\foreach \y in {-16} 
{
\draw [fill=gray!20] (\x+0,1+\y)--(\x+1,1+\y)--(\x+1,0+\y)--(\x+0,0+\y)--(\x+0,1+\y);
}
\foreach \x in {1,6}
\foreach \y in {-15} 
{
\draw [fill=gray!20] (\x+0,1+\y)--(\x+1,1+\y)--(\x+1,0+\y)--(\x+0,0+\y)--(\x+0,1+\y);
}

\foreach \x in {10,...,17}
\foreach \y in {-13,...,-6} 
{
\draw [fill=gray!20] (\x+0,1+\y)--(\x+1,1+\y)--(\x+1,0+\y)--(\x+0,0+\y)--(\x+0,1+\y);
}
\foreach \x in {11,...,16}
\foreach \y in {-15,-16} 
{
\draw [fill=gray!20] (\x+0,1+\y)--(\x+1,1+\y)--(\x+1,0+\y)--(\x+0,0+\y)--(\x+0,1+\y);
}

\foreach \x in {20,...,27}
\foreach \y in {-13,...,-6} 
{
\draw [fill=gray!20] (\x+0,1+\y)--(\x+1,1+\y)--(\x+1,0+\y)--(\x+0,0+\y)--(\x+0,1+\y);
}

\end{tikzpicture}
\end{center}
\caption{Level-1 layer diagram of $C$, a functional cube with a bump.}\label{fig_3d_C}
\end{figure}
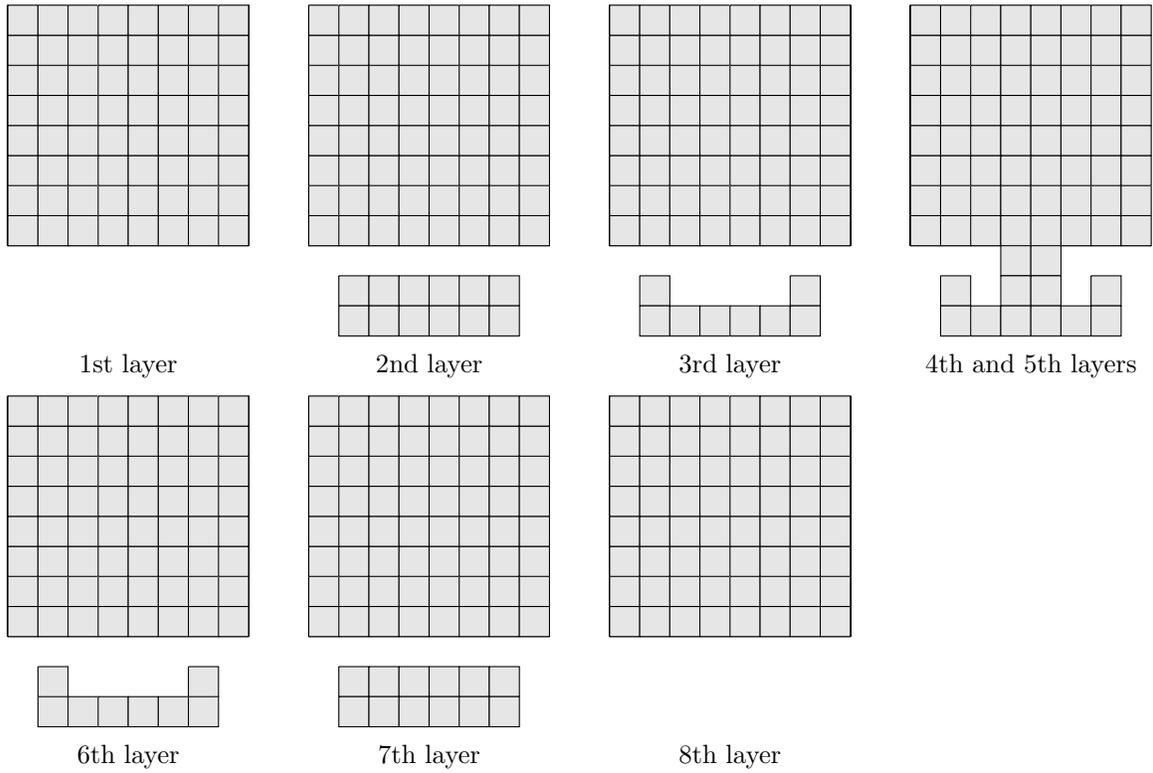

The second pair of building blocks, $d$ and $d^-$ are also used for encoding the colors of Wang tiles. The building block $d$ and $d^-$ can be obtained from $c$ and $c^-$ by a rotation of $180^{\circ}$ about a vertical axis, respectively. The building block $D^+$ is constructed from a $8\times 9\times 8$ polycube by adding a bump on its north side that matches the dent of $d$ or $d^-$. Figure \ref{fig_3d_d} and Figure \ref{fig_3d_D_plus} illustrate the building blocks $d$ and $D^+$, respectively.


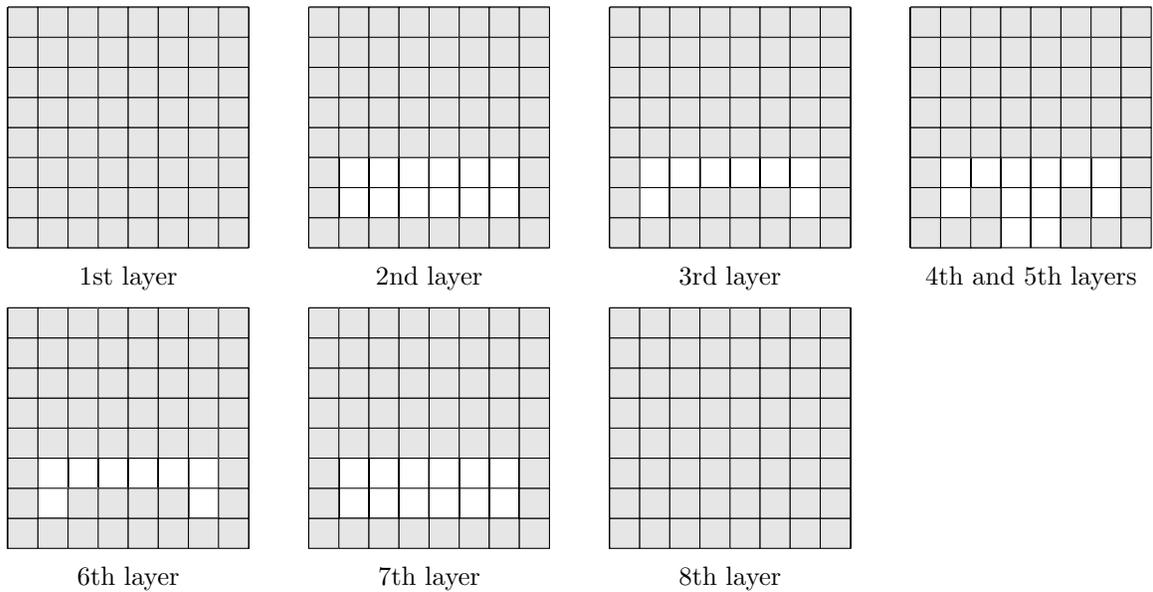
\begin{figure}[H]
\begin{center}
\begin{tikzpicture}[scale=0.4]

\foreach \x in {0,10,20,30}
\foreach \y in {0,...,8} 
{ 
\draw  (\x+0,0+\y)--(\x+8,0+\y);
\draw  (\x,0)--(\x,8);
\draw  (\x+8,0)--(\x+8,8);
\draw  (\x+7,0)--(\x+7,8);
\draw  (\x+6,0)--(\x+6,8);
\draw  (\x+5,0)--(\x+5,8);
\draw  (\x+4,0)--(\x+4,8);
\draw  (\x+3,0)--(\x+3,8);
\draw  (\x+2,0)--(\x+2,8);
\draw  (\x+1,0)--(\x+1,8);
}

\foreach \x in {0,...,7}
\foreach \y in {0,...,7} 
{
\draw [fill=gray!20] (\x+0,1+\y)--(\x+1,1+\y)--(\x+1,0+\y)--(\x+0,0+\y)--(\x+0,1+\y);
}

\foreach \x in {10,...,17}
\foreach \y in {0,3,4,5,6,7} 
{
\draw [fill=gray!20] (\x+0,1+\y)--(\x+1,1+\y)--(\x+1,0+\y)--(\x+0,0+\y)--(\x+0,1+\y);
}
\foreach \x in {10,17}
\foreach \y in {1,2} 
{
\draw [fill=gray!20] (\x+0,1+\y)--(\x+1,1+\y)--(\x+1,0+\y)--(\x+0,0+\y)--(\x+0,1+\y);
}

\foreach \x in {20,...,27}
\foreach \y in {0,3,4,5,6,7} 
{
\draw [fill=gray!20] (\x+0,1+\y)--(\x+1,1+\y)--(\x+1,0+\y)--(\x+0,0+\y)--(\x+0,1+\y);
}
\foreach \x in {20,27}
\foreach \y in {1,2} 
{
\draw [fill=gray!20] (\x+0,1+\y)--(\x+1,1+\y)--(\x+1,0+\y)--(\x+0,0+\y)--(\x+0,1+\y);
}
\foreach \x in {22,...,25}
\foreach \y in {1} 
{
\draw [fill=gray!20] (\x+0,1+\y)--(\x+1,1+\y)--(\x+1,0+\y)--(\x+0,0+\y)--(\x+0,1+\y);
}

\foreach \x in {30,...,37}
\foreach \y in {3,...,7} 
{
\draw [fill=gray!20] (\x+0,1+\y)--(\x+1,1+\y)--(\x+1,0+\y)--(\x+0,0+\y)--(\x+0,1+\y);
}
\foreach \x in {30,37}
\foreach \y in {1,2} 
{
\draw [fill=gray!20] (\x+0,1+\y)--(\x+1,1+\y)--(\x+1,0+\y)--(\x+0,0+\y)--(\x+0,1+\y);
}
\foreach \x in {32,35}
\foreach \y in {1} 
{
\draw [fill=gray!20] (\x+0,1+\y)--(\x+1,1+\y)--(\x+1,0+\y)--(\x+0,0+\y)--(\x+0,1+\y);
}
\foreach \x in {30,31,32,35,36,37}
\foreach \y in {0} 
{
\draw [fill=gray!20] (\x+0,1+\y)--(\x+1,1+\y)--(\x+1,0+\y)--(\x+0,0+\y)--(\x+0,1+\y);
}

\node at (4,-1) {$1$st layer};  \node at (14,-1) {$2$nd layer};  \node at (24,-1) {$3$rd layer};  \node at (34,-1) {$4$th and 5th layers};

\node at (4,-11) {$6$th layer};  \node at (14,-11) {$7$th layer};  \node at (24,-11) {$8$th layer};


\foreach \x in {0,10,20}
\foreach \y in {-10,...,-2} 
{ 
\draw  (\x+0,0+\y)--(\x+8,0+\y);
\draw  (\x,-10)--(\x,-2);
\draw  (\x+8,-10)--(\x+8,-2);
\draw  (\x+7,-10)--(\x+7,-2);
\draw  (\x+6,-10)--(\x+6,-2);
\draw  (\x+5,-10)--(\x+5,-2);
\draw  (\x+4,-10)--(\x+4,-2);
\draw  (\x+3,-10)--(\x+3,-2);
\draw  (\x+2,-10)--(\x+2,-2);
\draw  (\x+1,-10)--(\x+1,-2);
}

\foreach \x in {0,...,7}
\foreach \y in {-10,-7,-6,-5,-4,-3} 
{
\draw [fill=gray!20] (\x+0,1+\y)--(\x+1,1+\y)--(\x+1,0+\y)--(\x+0,0+\y)--(\x+0,1+\y);
}
\foreach \x in {0,7}
\foreach \y in {-9,-8} 
{
\draw [fill=gray!20] (\x+0,1+\y)--(\x+1,1+\y)--(\x+1,0+\y)--(\x+0,0+\y)--(\x+0,1+\y);
}
\foreach \x in {2,...,5}
\foreach \y in {-9} 
{
\draw [fill=gray!20] (\x+0,1+\y)--(\x+1,1+\y)--(\x+1,0+\y)--(\x+0,0+\y)--(\x+0,1+\y);
}

\foreach \x in {10,...,17}
\foreach \y in {-10,-7,-6,-5,-4,-3} 
{
\draw [fill=gray!20] (\x+0,1+\y)--(\x+1,1+\y)--(\x+1,0+\y)--(\x+0,0+\y)--(\x+0,1+\y);
}
\foreach \x in {10,17}
\foreach \y in {-9,-8} 
{
\draw [fill=gray!20] (\x+0,1+\y)--(\x+1,1+\y)--(\x+1,0+\y)--(\x+0,0+\y)--(\x+0,1+\y);
}

\foreach \x in {20,...,27}
\foreach \y in {-10,...,-3} 
{
\draw [fill=gray!20] (\x+0,1+\y)--(\x+1,1+\y)--(\x+1,0+\y)--(\x+0,0+\y)--(\x+0,1+\y);
}

\end{tikzpicture}
\end{center}
\caption{Level-1 layer diagram of $d$, a functional cube with a dent.}\label{fig_3d_d}
\end{figure}


\begin{figure}[H]
\begin{center}
\begin{tikzpicture}[scale=0.4]

\foreach \x in {0,...,7}
\foreach \y in {0,...,8} 
{
\draw [fill=gray!20] (\x+0,1+\y)--(\x+1,1+\y)--(\x+1,0+\y)--(\x+0,0+\y)--(\x+0,1+\y);
}

\foreach \x in {10,...,17}
\foreach \y in {0,...,8} 
{
\draw [fill=gray!20] (\x+0,1+\y)--(\x+1,1+\y)--(\x+1,0+\y)--(\x+0,0+\y)--(\x+0,1+\y);
}
\foreach \x in {11,...,16}
\foreach \y in {10,11} 
{
\draw [fill=gray!20] (\x+0,1+\y)--(\x+1,1+\y)--(\x+1,0+\y)--(\x+0,0+\y)--(\x+0,1+\y);
}

\foreach \x in {20,...,27}
\foreach \y in {0,...,8} 
{
\draw [fill=gray!20] (\x+0,1+\y)--(\x+1,1+\y)--(\x+1,0+\y)--(\x+0,0+\y)--(\x+0,1+\y);
}
\foreach \x in {21,...,26}
\foreach \y in {11} 
{
\draw [fill=gray!20] (\x+0,1+\y)--(\x+1,1+\y)--(\x+1,0+\y)--(\x+0,0+\y)--(\x+0,1+\y);
}
\foreach \x in {21,26}
\foreach \y in {10} 
{
\draw [fill=gray!20] (\x+0,1+\y)--(\x+1,1+\y)--(\x+1,0+\y)--(\x+0,0+\y)--(\x+0,1+\y);
}

\foreach \x in {30,...,37}
\foreach \y in {0,...,8} 
{
\draw [fill=gray!20] (\x+0,1+\y)--(\x+1,1+\y)--(\x+1,0+\y)--(\x+0,0+\y)--(\x+0,1+\y);
}
\foreach \x in {31,...,36}
\foreach \y in {11} 
{
\draw [fill=gray!20] (\x+0,1+\y)--(\x+1,1+\y)--(\x+1,0+\y)--(\x+0,0+\y)--(\x+0,1+\y);
}
\foreach \x in {31,36}
\foreach \y in {10} 
{
\draw [fill=gray!20] (\x+0,1+\y)--(\x+1,1+\y)--(\x+1,0+\y)--(\x+0,0+\y)--(\x+0,1+\y);
}
\foreach \x in {33,34}
\foreach \y in {9,10} 
{
\draw [fill=gray!20] (\x+0,1+\y)--(\x+1,1+\y)--(\x+1,0+\y)--(\x+0,0+\y)--(\x+0,1+\y);
}

\foreach \x in {0,10,20,30}
\foreach \y in {8} 
{
\draw [color=red, thick] (\x,\y)--(\x+8,\y);
}

\node at (4,-1) {$1$st layer};  \node at (14,-1) {$2$nd layer};  \node at (24,-1) {$3$rd layer};  \node at (34,-1) {$4$th and 5th layers};

\node at (4,-15) {$6$th layer};  \node at (14,-15) {$7$th layer};  \node at (24,-15) {$8$th layer};

\foreach \x in {0,...,7}
\foreach \y in {-14,...,-6} 
{
\draw [fill=gray!20] (\x+0,1+\y)--(\x+1,1+\y)--(\x+1,0+\y)--(\x+0,0+\y)--(\x+0,1+\y);
}
\foreach \x in {1,...,6}
\foreach \y in {-3} 
{
\draw [fill=gray!20] (\x+0,1+\y)--(\x+1,1+\y)--(\x+1,0+\y)--(\x+0,0+\y)--(\x+0,1+\y);
}
\foreach \x in {1,6}
\foreach \y in {-4} 
{
\draw [fill=gray!20] (\x+0,1+\y)--(\x+1,1+\y)--(\x+1,0+\y)--(\x+0,0+\y)--(\x+0,1+\y);
}

\foreach \x in {10,...,17}
\foreach \y in {-14,...,-6} 
{
\draw [fill=gray!20] (\x+0,1+\y)--(\x+1,1+\y)--(\x+1,0+\y)--(\x+0,0+\y)--(\x+0,1+\y);
}
\foreach \x in {11,...,16}
\foreach \y in {-3,-4} 
{
\draw [fill=gray!20] (\x+0,1+\y)--(\x+1,1+\y)--(\x+1,0+\y)--(\x+0,0+\y)--(\x+0,1+\y);
}

\foreach \x in {20,...,27}
\foreach \y in {-14,...,-6} 
{
\draw [fill=gray!20] (\x+0,1+\y)--(\x+1,1+\y)--(\x+1,0+\y)--(\x+0,0+\y)--(\x+0,1+\y);
}

\foreach \x in {0,10,20}
\foreach \y in {-6} 
{
\draw [color=red, thick] (\x,\y)--(\x+8,\y);
}

\end{tikzpicture}
\end{center}
\caption{Level-1 layer diagram of $D^+$, a functional cube with a bump.}\label{fig_3d_D_plus}
\end{figure}
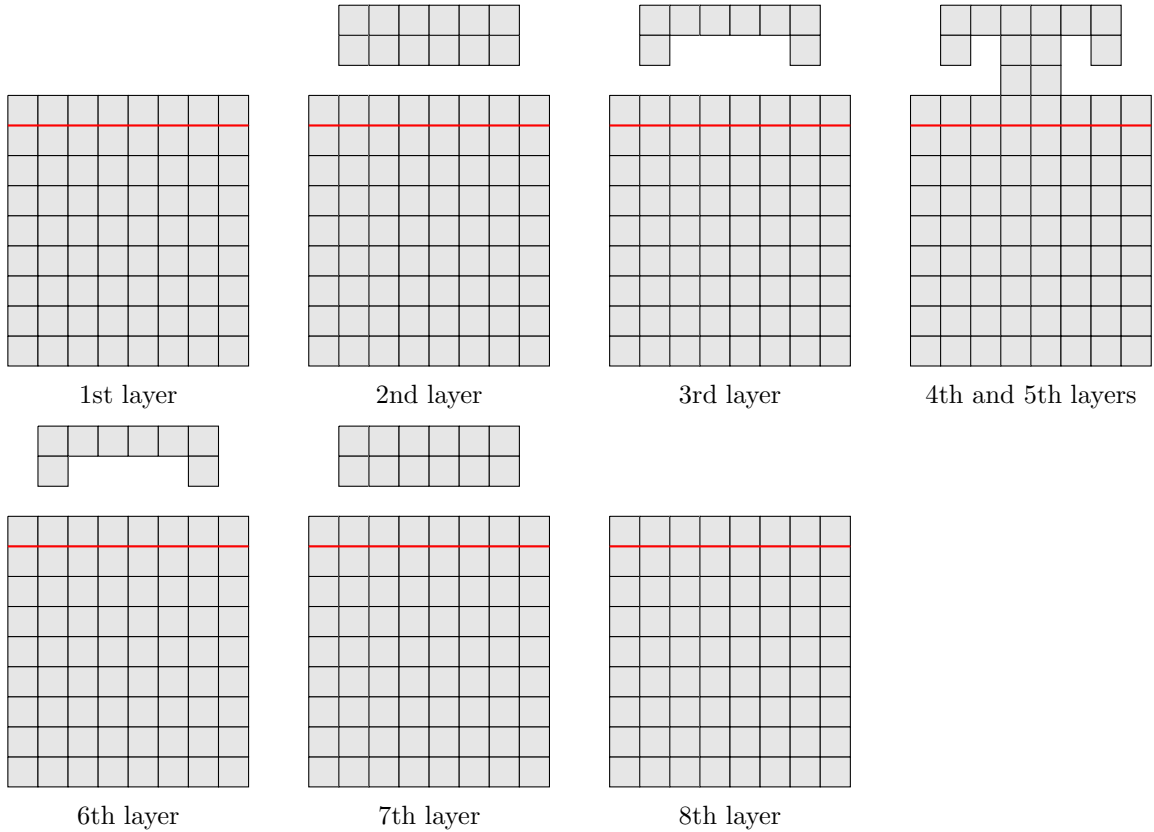

The building block $a$ and $b$ are almost the same as the building block $c$, except that the dents of $a$ and $b$ are on their east side and west side, respectively. In other words, the building block $a$ (see Figure \ref{fig_3d_a}) is obtained from $c$ by a rotation of $90^{\circ}$ clockwise about a vertical axis, and $b$ is obtained from $c$ by a rotation of $90^{\circ}$ counterclockwise about a vertical axis. Similarly, building blocks $A$ and $B$ are obtained from $C$ by a rotation of $90^{\circ}$ clockwise and counterclockwise about a vertical axis, respectively.

The two paris of building blocks: $a$ and $A$, and $b$ and $B$ are used to ensure that the encoders (which will be introduced in the next subsection) must be placed inside selectors (see next subsections too) in any tilings.

Building block $x$ (see Figure \ref{fig_3d_x}) is a functional cube with a dent on its west side. The shape of the dent of $x$ is slightly different than those of $a$, $b$ and $c$. Building block $y$ is obtained from $x$ by a rotation of $90^{\circ}$ counterclockwise about a vertical axis. So the dent of $y$ is on its south side. Building block $z$ is obtained from $x$ by a rotation of $90^{\circ}$ about a south-north axis so that the dent is on the bottom side of $z$. Building blocks $X$, $Y$ and $Z$ are functional cubes with a bump that matches the dent of $x$, $y$ and $z$, respectively.

The there pairs of building blocks: $x$ and $X$, $y$ and $Y$, and $z$ and $Z$ are used to fixed the relative positions in any tilings between the polycubes called selectors which will be introduced in the next subsection.

The last two building blocks, which are also two tiles in the set of $5$ tiles in the proof of Theorem \ref{thm_3d_new}, are illustrated in Figure \ref{fig_3d_f} and Figure \ref{fig_3d_f_plus}. They are called the \textit{standard filler} $F$ and \textit{bigger filler} $F^+$, respectively, as the latter is longer than the former in the south-north direction (plus a vertical $8\times 1\times 8$  layer in the middle).


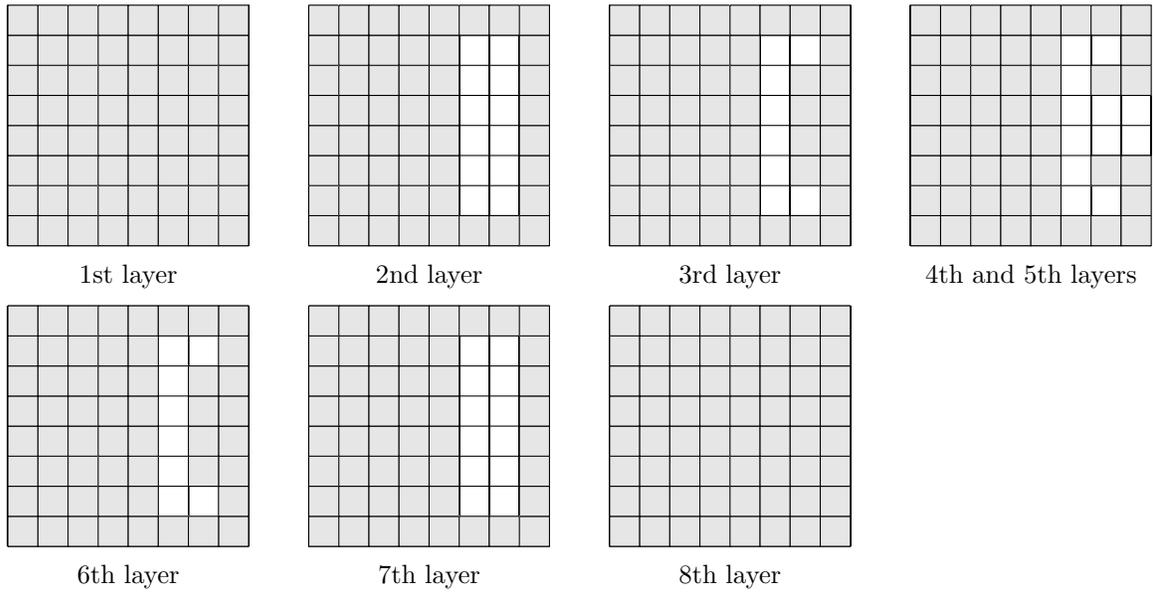
\begin{figure}[H]
\begin{center}
\begin{tikzpicture}[scale=0.4]

\foreach \x in {0,10,20,30}
\foreach \y in {0,...,8} 
{ 
\draw  (\x+0,0+\y)--(\x+8,0+\y);
\draw  (\x,0)--(\x,8);
\draw  (\x+8,0)--(\x+8,8);
\draw  (\x+7,0)--(\x+7,8);
\draw  (\x+6,0)--(\x+6,8);
\draw  (\x+5,0)--(\x+5,8);
\draw  (\x+4,0)--(\x+4,8);
\draw  (\x+3,0)--(\x+3,8);
\draw  (\x+2,0)--(\x+2,8);
\draw  (\x+1,0)--(\x+1,8);
}

\foreach \x in {0,...,7}
\foreach \y in {0,...,7} 
{
\draw [fill=gray!20] (\x+0,1+\y)--(\x+1,1+\y)--(\x+1,0+\y)--(\x+0,0+\y)--(\x+0,1+\y);
}

\foreach \x in {10,11,12,13,14,17}
\foreach \y in {0,...,7} 
{
\draw [fill=gray!20] (\x+0,1+\y)--(\x+1,1+\y)--(\x+1,0+\y)--(\x+0,0+\y)--(\x+0,1+\y);
}
\foreach \x in {15,16}
\foreach \y in {0,7} 
{
\draw [fill=gray!20] (\x+0,1+\y)--(\x+1,1+\y)--(\x+1,0+\y)--(\x+0,0+\y)--(\x+0,1+\y);
}

\foreach \x in {20,21,22,23,24,27}
\foreach \y in {0,...,7} 
{
\draw [fill=gray!20] (\x+0,1+\y)--(\x+1,1+\y)--(\x+1,0+\y)--(\x+0,0+\y)--(\x+0,1+\y);
}
\foreach \x in {25,26}
\foreach \y in {0,7} 
{
\draw [fill=gray!20] (\x+0,1+\y)--(\x+1,1+\y)--(\x+1,0+\y)--(\x+0,0+\y)--(\x+0,1+\y);
}
\foreach \x in {26}
\foreach \y in {2,...,5} 
{
\draw [fill=gray!20] (\x+0,1+\y)--(\x+1,1+\y)--(\x+1,0+\y)--(\x+0,0+\y)--(\x+0,1+\y);
}

\foreach \x in {30,31,32,33,34}
\foreach \y in {0,...,7} 
{
\draw [fill=gray!20] (\x+0,1+\y)--(\x+1,1+\y)--(\x+1,0+\y)--(\x+0,0+\y)--(\x+0,1+\y);
}
\foreach \x in {35,36}
\foreach \y in {0,7} 
{
\draw [fill=gray!20] (\x+0,1+\y)--(\x+1,1+\y)--(\x+1,0+\y)--(\x+0,0+\y)--(\x+0,1+\y);
}
\foreach \x in {36}
\foreach \y in {2,5} 
{
\draw [fill=gray!20] (\x+0,1+\y)--(\x+1,1+\y)--(\x+1,0+\y)--(\x+0,0+\y)--(\x+0,1+\y);
}
\foreach \x in {37}
\foreach \y in {0,1,2,5,6,7} 
{
\draw [fill=gray!20] (\x+0,1+\y)--(\x+1,1+\y)--(\x+1,0+\y)--(\x+0,0+\y)--(\x+0,1+\y);
}

\node at (4,-1) {$1$st layer};  \node at (14,-1) {$2$nd layer};  \node at (24,-1) {$3$rd layer};  \node at (34,-1) {$4$th and 5th layers};

\node at (4,-11) {$6$th layer};  \node at (14,-11) {$7$th layer};  \node at (24,-11) {$8$th layer};


\foreach \x in {0,10,20}
\foreach \y in {-10,...,-2} 
{ 
\draw  (\x+0,0+\y)--(\x+8,0+\y);
\draw  (\x,-10)--(\x,-2);
\draw  (\x+8,-10)--(\x+8,-2);
\draw  (\x+7,-10)--(\x+7,-2);
\draw  (\x+6,-10)--(\x+6,-2);
\draw  (\x+5,-10)--(\x+5,-2);
\draw  (\x+4,-10)--(\x+4,-2);
\draw  (\x+3,-10)--(\x+3,-2);
\draw  (\x+2,-10)--(\x+2,-2);
\draw  (\x+1,-10)--(\x+1,-2);
}

\foreach \x in {0,1,2,3,4,7}
\foreach \y in {-10,...,-3} 
{
\draw [fill=gray!20] (\x+0,1+\y)--(\x+1,1+\y)--(\x+1,0+\y)--(\x+0,0+\y)--(\x+0,1+\y);
}
\foreach \x in {5,6}
\foreach \y in {-3,-10} 
{
\draw [fill=gray!20] (\x+0,1+\y)--(\x+1,1+\y)--(\x+1,0+\y)--(\x+0,0+\y)--(\x+0,1+\y);
}
\foreach \x in {6}
\foreach \y in {-8,...,-5} 
{
\draw [fill=gray!20] (\x+0,1+\y)--(\x+1,1+\y)--(\x+1,0+\y)--(\x+0,0+\y)--(\x+0,1+\y);
}

\foreach \x in {10,11,12,13,14,17}
\foreach \y in {-10,...,-3} 
{
\draw [fill=gray!20] (\x+0,1+\y)--(\x+1,1+\y)--(\x+1,0+\y)--(\x+0,0+\y)--(\x+0,1+\y);
}
\foreach \x in {15,16}
\foreach \y in {-3,-10} 
{
\draw [fill=gray!20] (\x+0,1+\y)--(\x+1,1+\y)--(\x+1,0+\y)--(\x+0,0+\y)--(\x+0,1+\y);
}

\foreach \x in {20,...,27}
\foreach \y in {-10,...,-3} 
{
\draw [fill=gray!20] (\x+0,1+\y)--(\x+1,1+\y)--(\x+1,0+\y)--(\x+0,0+\y)--(\x+0,1+\y);
}

\end{tikzpicture}
\end{center}
\caption{Level-1 layer diagram of $a$, a functional cube with a dent.}\label{fig_3d_a}
\end{figure}


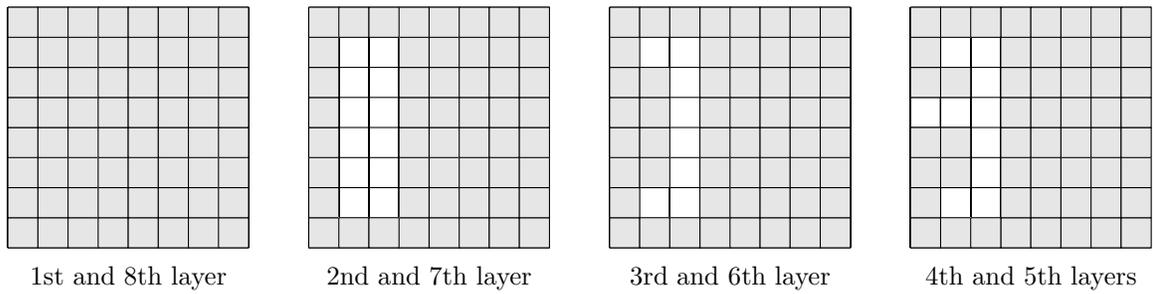
\begin{figure}[H]
\begin{center}
\begin{tikzpicture}[scale=0.4]

\foreach \x in {0,10,20,30}
\foreach \y in {0,...,8} 
{ 
\draw  (\x+0,0+\y)--(\x+8,0+\y);
\draw  (\x,0)--(\x,8);
\draw  (\x+8,0)--(\x+8,8);
\draw  (\x+7,0)--(\x+7,8);
\draw  (\x+6,0)--(\x+6,8);
\draw  (\x+5,0)--(\x+5,8);
\draw  (\x+4,0)--(\x+4,8);
\draw  (\x+3,0)--(\x+3,8);
\draw  (\x+2,0)--(\x+2,8);
\draw  (\x+1,0)--(\x+1,8);
}

\foreach \x in {0,...,7}
\foreach \y in {0,...,7} 
{
\draw [fill=gray!20] (\x+0,1+\y)--(\x+1,1+\y)--(\x+1,0+\y)--(\x+0,0+\y)--(\x+0,1+\y);
}

\foreach \x in {10,13,14,15,16,17}
\foreach \y in {0,...,7} 
{
\draw [fill=gray!20] (\x+0,1+\y)--(\x+1,1+\y)--(\x+1,0+\y)--(\x+0,0+\y)--(\x+0,1+\y);
}
\foreach \x in {11,12}
\foreach \y in {0,7} 
{
\draw [fill=gray!20] (\x+0,1+\y)--(\x+1,1+\y)--(\x+1,0+\y)--(\x+0,0+\y)--(\x+0,1+\y);
}

\foreach \x in {20,23,24,25,26,27}
\foreach \y in {0,...,7} 
{
\draw [fill=gray!20] (\x+0,1+\y)--(\x+1,1+\y)--(\x+1,0+\y)--(\x+0,0+\y)--(\x+0,1+\y);
}
\foreach \x in {21,22}
\foreach \y in {0,7} 
{
\draw [fill=gray!20] (\x+0,1+\y)--(\x+1,1+\y)--(\x+1,0+\y)--(\x+0,0+\y)--(\x+0,1+\y);
}
\foreach \x in {21}
\foreach \y in {2,...,5} 
{
\draw [fill=gray!20] (\x+0,1+\y)--(\x+1,1+\y)--(\x+1,0+\y)--(\x+0,0+\y)--(\x+0,1+\y);
}

\foreach \x in {33,34,35,36,37}
\foreach \y in {0,...,7} 
{
\draw [fill=gray!20] (\x+0,1+\y)--(\x+1,1+\y)--(\x+1,0+\y)--(\x+0,0+\y)--(\x+0,1+\y);
}
\foreach \x in {31,32}
\foreach \y in {0,7} 
{
\draw [fill=gray!20] (\x+0,1+\y)--(\x+1,1+\y)--(\x+1,0+\y)--(\x+0,0+\y)--(\x+0,1+\y);
}
\foreach \x in {31}
\foreach \y in {2,3,5} 
{
\draw [fill=gray!20] (\x+0,1+\y)--(\x+1,1+\y)--(\x+1,0+\y)--(\x+0,0+\y)--(\x+0,1+\y);
}
\foreach \x in {30}
\foreach \y in {0,1,2,3,5,6,7} 
{
\draw [fill=gray!20] (\x+0,1+\y)--(\x+1,1+\y)--(\x+1,0+\y)--(\x+0,0+\y)--(\x+0,1+\y);
}

\node at (4,-1) {$1$st and $8$th layer};  \node at (14,-1) {$2$nd and $7$th layer};  \node at (24,-1) {$3$rd and $6$th layer};  \node at (34,-1) {$4$th and 5th layers};

\end{tikzpicture}
\end{center}
\caption{Level-1 layer diagram of $x$, a functional cube with a dent.}\label{fig_3d_x}
\end{figure}


\begin{figure}[H]
\begin{center}
\begin{tikzpicture}[scale=0.4]

\foreach \x in {1,...,6}
\foreach \y in {1,2,5,6} 
{
\draw [fill=gray!20] (\x+0,1+\y)--(\x+1,1+\y)--(\x+1,0+\y)--(\x+0,0+\y)--(\x+0,1+\y);
}

\foreach \x in {11,...,16}
\foreach \y in {1,6} 
{
\draw [fill=gray!20] (\x+0,1+\y)--(\x+1,1+\y)--(\x+1,0+\y)--(\x+0,0+\y)--(\x+0,1+\y);
}
\foreach \x in {11,16}
\foreach \y in {2,5} 
{
\draw [fill=gray!20] (\x+0,1+\y)--(\x+1,1+\y)--(\x+1,0+\y)--(\x+0,0+\y)--(\x+0,1+\y);
}

\foreach \x in {21,...,26}
\foreach \y in {1,6} 
{
\draw [fill=gray!20] (\x+0,1+\y)--(\x+1,1+\y)--(\x+1,0+\y)--(\x+0,0+\y)--(\x+0,1+\y);
}
\foreach \x in {21,26}
\foreach \y in {2,5} 
{
\draw [fill=gray!20] (\x+0,1+\y)--(\x+1,1+\y)--(\x+1,0+\y)--(\x+0,0+\y)--(\x+0,1+\y);
}
\foreach \x in {23,24}
\foreach \y in {2,...,5} 
{
\draw [fill=gray!20] (\x+0,1+\y)--(\x+1,1+\y)--(\x+1,0+\y)--(\x+0,0+\y)--(\x+0,1+\y);
}

\node at (4,0) {$1$st and $6$th layer};  \node at (14,0) {$2$nd and $5$th layer};  \node at (24,0) {$3$rd and $4$th layer};

\end{tikzpicture}
\end{center}
\caption{Level-1 layer diagram of the standard filler $F$.}\label{fig_3d_f}
\end{figure}
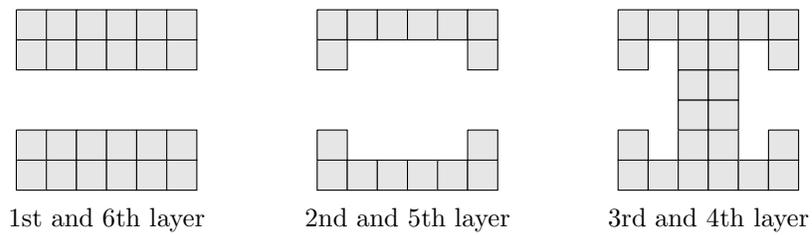


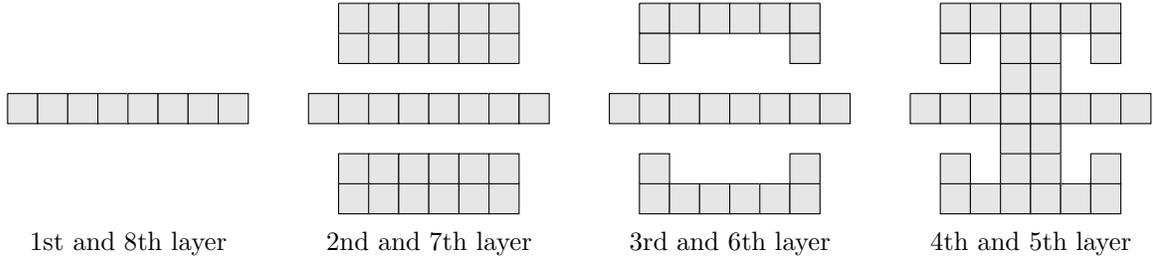
\begin{figure}[H]
\begin{center}
\begin{tikzpicture}[scale=0.4]

\foreach \x in {1,...,6}
\foreach \y in {0,1,5,6} 
{
\draw [fill=gray!20] (\x+0,1+\y)--(\x+1,1+\y)--(\x+1,0+\y)--(\x+0,0+\y)--(\x+0,1+\y);
}

\foreach \x in {11,...,16}
\foreach \y in {0,6} 
{
\draw [fill=gray!20] (\x+0,1+\y)--(\x+1,1+\y)--(\x+1,0+\y)--(\x+0,0+\y)--(\x+0,1+\y);
}
\foreach \x in {11,16}
\foreach \y in {1,5} 
{
\draw [fill=gray!20] (\x+0,1+\y)--(\x+1,1+\y)--(\x+1,0+\y)--(\x+0,0+\y)--(\x+0,1+\y);
}

\foreach \x in {21,...,26}
\foreach \y in {0,6} 
{
\draw [fill=gray!20] (\x+0,1+\y)--(\x+1,1+\y)--(\x+1,0+\y)--(\x+0,0+\y)--(\x+0,1+\y);
}
\foreach \x in {21,26}
\foreach \y in {1,5} 
{
\draw [fill=gray!20] (\x+0,1+\y)--(\x+1,1+\y)--(\x+1,0+\y)--(\x+0,0+\y)--(\x+0,1+\y);
}
\foreach \x in {23,24}
\foreach \y in {1,...,5} 
{
\draw [fill=gray!20] (\x+0,1+\y)--(\x+1,1+\y)--(\x+1,0+\y)--(\x+0,0+\y)--(\x+0,1+\y);
}

\node at (4,-1) {$2$nd and $7$th layer};  \node at (14,-1) {$3$rd and $6$th layer};  \node at (24,-1) {$4$th and $5$th layer};

\node at (-6,-1) {$1$st and $8$th layer};  
\foreach \x in {-10,...,-3}
\foreach \y in {3} 
{
\draw [fill=gray!20] (\x+0,1+\y)--(\x+1,1+\y)--(\x+1,0+\y)--(\x+0,0+\y)--(\x+0,1+\y);
}

\foreach \x in {0,...,7}
\foreach \y in {3} 
{
\draw [fill=gray!20] (\x+0,1+\y)--(\x+1,1+\y)--(\x+1,0+\y)--(\x+0,0+\y)--(\x+0,1+\y);
}

\foreach \x in {10,...,17}
\foreach \y in {3} 
{
\draw [fill=gray!20] (\x+0,1+\y)--(\x+1,1+\y)--(\x+1,0+\y)--(\x+0,0+\y)--(\x+0,1+\y);
}

\foreach \x in {20,21,22,25,26,27}
\foreach \y in {3} 
{
\draw [fill=gray!20] (\x+0,1+\y)--(\x+1,1+\y)--(\x+1,0+\y)--(\x+0,0+\y)--(\x+0,1+\y);
}

\end{tikzpicture}
\end{center}
\caption{Level-1 layer diagram of the bigger filler $F^+$.}\label{fig_3d_f_plus}
\end{figure}

\subsection{The Set of $5$ Polycubes}

Recall that all the layer diagrams in the previous subsection are level-1, where each gray square represents a $1\times 1\times 1$ unit cube. To illustrate the set of $5$ polycubes, all layer diagrams in this subsection are \textit{level-2}, where each gray square represents a building block (i.e. a $8\times 8\times 8$ functional cube, or a functional cube with a dent or/and a bump).

We take the set of $3$ Wang tiles illustrated in Figure \ref{fig_w3} as an example to introduce the method to construct a corresponding set of $5$ polycubes. Two of them, the standard filler $F$ and bigger filler $F^+$, are already introduced in previous subsection. The remain three tiles are described in this subsection.

All the Wang tiles are simulated in a single polycube called \textit{encoder} as illustrated in Figure \ref{fig_encoder}. Each Wang tile is simulated in one layer (of level-2), so the encoder in Figure \ref{fig_encoder} has $3$ layers which simulate the $3$ Wang tiles in Figure \ref{fig_w3}, respectively. For the example in Figure \ref{fig_w3}, the colors of Wang tiles can be encoded by $2$-bit binary strings, as there are $4$ colors in total. Let the colors red, green, blue, and yellow be encoded in binary strings 00, 01, 10, and 11, respectively. On the north side of a layer, 0 and 1 are represented by building blocks $c$ and $c^-$, respectively. On the south side of a layer, 0 and 1 are represented by $d^-$ and $d$, respectively. In other words, $0$ is represented by $c$ on the north and $d^-$ on the south, and $1$ is represented by $c^-$ on the north and $d$ on the south. In general, to simulate a set of $p$ Wang tiles with $q$ different colors, we construct an encoder with $p$ layers, with each layer consisting of $(2t+2)\times (2t+2)$ building blocks, where $t=\lceil \log_2 q\rceil$.


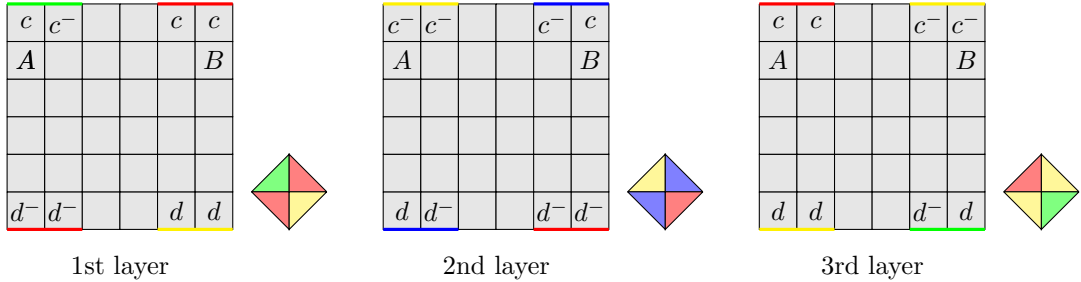
\begin{figure}[H]
\begin{center}
\begin{tikzpicture}[scale=0.5]

\foreach \x in {0,...,5}
\foreach \y in {0,...,5} 
{
\draw [ fill=gray!20] (\x+0,1+\y)--(\x+1,1+\y)--(\x+1,0+\y)--(\x+0,0+\y)--(\x+0,1+\y);
}

\foreach \x in {10,...,15}
\foreach \y in {0,...,5} 
{
\draw [ fill=gray!20] (\x+0,1+\y)--(\x+1,1+\y)--(\x+1,0+\y)--(\x+0,0+\y)--(\x+0,1+\y);
}

\foreach \x in {20,...,25}
\foreach \y in {0,...,5} 
{
\draw [ fill=gray!20] (\x+0,1+\y)--(\x+1,1+\y)--(\x+1,0+\y)--(\x+0,0+\y)--(\x+0,1+\y);
}

\foreach \x in {0,4,5,15,20,21}
\foreach \y in {5} 
{
\node at (\x+0.5,\y+0.5) {$c$};
}
\foreach \x in {1,10,11,14,24,25}
\foreach \y in {5} 
{
\node at (\x+0.5,\y+0.5) {$c^-$};
}

\foreach \x in {4,5,10,20,21,25}
\foreach \y in {0} 
{
\node at (\x+0.5,\y+0.5) {$d$};
}

\foreach \x in {0,1,11,14,15,24}
\foreach \y in {0} 
{
\node at (\x+0.5,\y+0.5) {$d^-$};
}

\foreach \x in {0,10,,20}
\foreach \y in {4} 
{
\node at (\x+0.5,\y+0.5) {$A$};
}
\foreach \x in {5,15,25}
\foreach \y in {4} 
{
\node at (\x+0.5,\y+0.5) {$B$};
}

\node at (3,-1) {$1$st layer};  \node at (13,-1) {$2$nd layer};  \node at (23,-1) {$3$rd layer};   

\draw [very thick, color=green] (0,6)--(2,6); 
\draw [very thick, color=green] (24,0)--(26,0);

\draw [very thick, color=red] (4,6)--(6,6); 
\draw [very thick, color=red] (0,0)--(2,0); 
\draw [very thick, color=red] (14,0)--(16,0); 
\draw [very thick, color=red] (20,6)--(22,6);

\draw [very thick, color=blue] (14,6)--(16,6); 
\draw [very thick, color=blue] (10,0)--(12,0);

\draw [very thick, color=yellow] (24,6)--(26,6); 
\draw [very thick, color=yellow] (20,0)--(22,0); 
\draw [very thick, color=yellow] (10,6)--(12,6); 
\draw [very thick, color=yellow] (4,0)--(6,0);

\draw [fill=green!50] (7-0.5,1)--(8-0.5,1)--(8-0.5,2)--(7-0.5,1);
\draw [fill=red!50] (9-0.5,1)--(8-0.5,1)--(8-0.5,2)--(9-0.5,1);
\draw [fill=red!50] (7-0.5,1)--(8-0.5,1)--(8-0.5,0)--(7-0.5,1);
\draw [fill=yellow!50] (9-0.5,1)--(8-0.5,1)--(8-0.5,0)--(9-0.5,1);

\draw [fill=yellow!50] (17-0.5,1)--(18-0.5,1)--(18-0.5,2)--(17-0.5,1);
\draw [fill=blue!50] (19-0.5,1)--(18-0.5,1)--(18-0.5,2)--(19-0.5,1);
\draw [fill=blue!50] (17-0.5,1)--(18-0.5,1)--(18-0.5,0)--(17-0.5,1);
\draw [fill=red!50] (19-0.5,1)--(18-0.5,1)--(18-0.5,0)--(19-0.5,1);

\draw [fill=red!50] (27-0.5,1)--(28-0.5,1)--(28-0.5,2)--(27-0.5,1);
\draw [fill=yellow!50] (29-0.5,1)--(28-0.5,1)--(28-0.5,2)--(29-0.5,1);
\draw [fill=yellow!50] (27-0.5,1)--(28-0.5,1)--(28-0.5,0)--(27-0.5,1);
\draw [fill=green!50] (29-0.5,1)--(28-0.5,1)--(28-0.5,0)--(29-0.5,1);

\end{tikzpicture}
\end{center}
\caption{Level-$2$ layer diagram of the encoder.}\label{fig_encoder}
\end{figure}

Besides the building blocks $c$, $c^-$, $d$, and $d^-$, there are also building blocks $A$ and $B$ on each layer of the encoder. The building blocks $A$ and $B$ ensure that the encoder must be placed inside the selector which will be introduced in the next paragraph.

The \textit{selector} is illustrated in Figure \ref{fig_selector}. Intuitively speaking, the selector is like a vertical well with the same number of layers as an encoder. There is a building block $z$ on the northeast corner of the first layer and a building block $Z$ on the northeast corner of the third layer of the selector. The unique shape of the dent of $z$ and bump of $Z$ will guarantee another selector to be placed above or below a selector in any tiling. Thus, they form a two-way infinite vertical well. There are also building blocks $x$, $X$, $y$, and $Y$ on the first layer of the selector. The building blocks $x$ and $X$ determine the relative position of adjacent selectors in the east-west direction in any tiling of the space. The building blocks $y$ and $Y$ determine the relative position of adjacent selectors in the south-north direction. Together, the selectors must form a pattern as illustrated later in Figure \ref{fig_3d_pattern} and Figure \ref{fig_3d_pattern_2} for every horizontal layers.


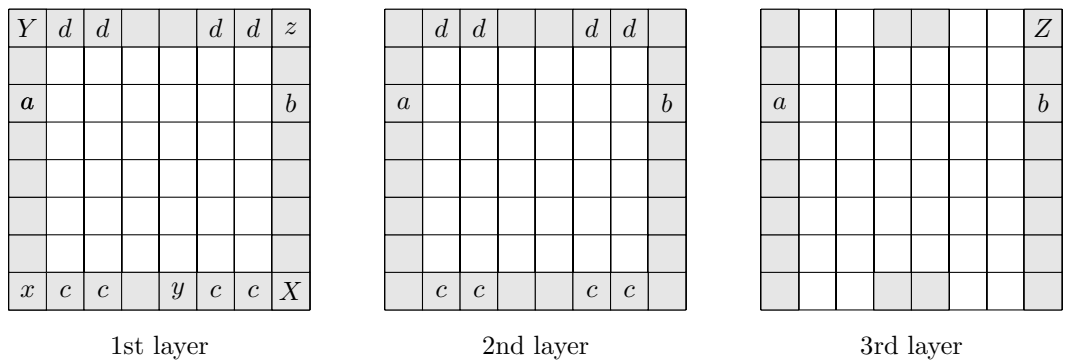
\begin{figure}[H]
\begin{center}
\begin{tikzpicture}[scale=0.5]

\foreach \x in {0,10,20}
\foreach \y in {0,...,8} 
{ 
\draw   (\x+0,0+\y)--(\x+8,0+\y);
\draw  (\x,0)--(\x,8);
\draw   (\x+8,0)--(\x+8,8);
\draw   (\x+7,0)--(\x+7,8);
\draw   (\x+6,0)--(\x+6,8);
\draw   (\x+5,0)--(\x+5,8);
\draw   (\x+4,0)--(\x+4,8);
\draw   (\x+3,0)--(\x+3,8);
\draw   (\x+2,0)--(\x+2,8);
\draw   (\x+1,0)--(\x+1,8);
}

\foreach \x in {0,...,7}
\foreach \y in {0,7} 
{
\draw [ fill=gray!20] (\x+0,1+\y)--(\x+1,1+\y)--(\x+1,0+\y)--(\x+0,0+\y)--(\x+0,1+\y);
}
\foreach \x in {0,7}
\foreach \y in {0,...,7} 
{
\draw [ fill=gray!20] (\x+0,1+\y)--(\x+1,1+\y)--(\x+1,0+\y)--(\x+0,0+\y)--(\x+0,1+\y);
}

\foreach \x in {10,17}
\foreach \y in {0,...,7} 
{
\draw [ fill=gray!20] (\x+0,1+\y)--(\x+1,1+\y)--(\x+1,0+\y)--(\x+0,0+\y)--(\x+0,1+\y);
}
\foreach \x in {11,...,16}
\foreach \y in {0,7} 
{
\draw [ fill=gray!20] (\x+0,1+\y)--(\x+1,1+\y)--(\x+1,0+\y)--(\x+0,0+\y)--(\x+0,1+\y);
}

\foreach \x in {20,27}
\foreach \y in {0,...,7} 
{
\draw [ fill=gray!20] (\x+0,1+\y)--(\x+1,1+\y)--(\x+1,0+\y)--(\x+0,0+\y)--(\x+0,1+\y);
}
\foreach \x in {23,24}
\foreach \y in {0,7} 
{
\draw [ fill=gray!20] (\x+0,1+\y)--(\x+1,1+\y)--(\x+1,0+\y)--(\x+0,0+\y)--(\x+0,1+\y);
}

\foreach \x in {7}
\foreach \y in {0} 
{
\node at (\x+0.5,\y+0.5) {$X$};
}
\foreach \x in {0}
\foreach \y in {0} 
{
\node at (\x+0.5,\y+0.5) {$x$};
}

\foreach \x in {0}
\foreach \y in {7} 
{
\node at (\x+0.5,\y+0.5) {$Y$};
}
\foreach \x in {4}
\foreach \y in {0} 
{
\node at (\x+0.5,\y+0.5) {$y$};
}

\foreach \x in {7}
\foreach \y in {7} 
{
\node at (\x+0.5,\y+0.5) {$z$};
}
\foreach \x in {27}
\foreach \y in {7} 
{
\node at (\x+0.5,\y+0.5) {$Z$};
}

\foreach \x in {0,10,,20}
\foreach \y in {5} 
{
\node at (\x+0.5,\y+0.5) {$a$};
}
\foreach \x in {7,17,27}
\foreach \y in {5} 
{
\node at (\x+0.5,\y+0.5) {$b$};
}

\foreach \x in {1,2,5,6,11,12,15,16}
\foreach \y in {0} 
{
\node at (\x+0.5,\y+0.5) {$c$};
}

\foreach \x in {1,2,5,6,11,12,15,16}
\foreach \y in {7} 
{
\node at (\x+0.5,\y+0.5) {$d$};
}

\node at (4,-1) {$1$st layer};  \node at (14,-1) {$2$nd layer};  \node at (24,-1) {$3$rd layer};

\end{tikzpicture}
\end{center}
\caption{Level-$2$ layer diagram of the selector.}\label{fig_selector}
\end{figure}

The building blocks $a$ and $b$ appear on every layer of the selector. They can only be matched by the bumps of $A$ and $B$ which appear exclusively on the encoder. Therefore, the encoders must be placed inside the vertical wells formed by selectors.

On the first and second layers of the selector, there are building blocks $c$ or $d$ which will be adjacent to the north or south sides of the encoding blocks (i.e. $c$, $c^-$, $d$ or $d^-$) of the encoder after they are placed inside the wells of selectors. Therefore, there are many isolated tiny vacant holes between the first two layers of selectors and the encoders. These holes can be filled exactly by the standard filler $F$ or bigger filler $F^+$ introduced in the previous subsection. On the third layer of the selector, the locations that are supposed to meet the encoding blocks of the encoder to the north or to the south are vacant. So these are the only windows for a simulated Wang tile of the encoder to see outside through the selector. 

In general, the size of the selector will change in a similar way as the encoder, with respect to the number of tiles and colors of a set of Wang tiles. To be precise, for a set of $p$ Wang tiles with $q$ colors ($t = \lceil \log_2 q \rceil$), the outer boundary of every layer is of size $(2t+4)\times (2t+4)$ (count by number of building blocks), and the selector has $p$ layers.

Finally, the simulated Wang tiles that are exposed outside the selectors will be connected by the \textit{linker} as illustrated in Figure \ref{fig_linker}. The linker consists of just two building blocks: a $D^+$ on the north and a $C$ on the south. In other words, the linker is a $8\times 17\times 8$ polycube with bumps on both north and south sides. Just like the fillers, its size is fixed and does not increase with the number of tiles or colors of a set of Wang tiles. 


\begin{figure}[H]
\begin{center}
\begin{tikzpicture}[scale=0.5]

\foreach \x in {0}
\foreach \y in {0,1} 
{
\draw [ fill=gray!20] (\x+0,1+\y)--(\x+1,1+\y)--(\x+1,0+\y)--(\x+0,0+\y)--(\x+0,1+\y);
}

\foreach \x in {0}
\foreach \y in {0} 
{
\node at (\x+0.5,\y+0.5) {$C$};
}
\foreach \x in {0}
\foreach \y in {1} 
{
\node at (\x+0.5,\y+0.5) {$D^+$};
}


\end{tikzpicture}
\end{center}
\caption{Level-$2$ layer diagram of the linker.}\label{fig_linker}
\end{figure}
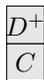

In summary, we have constructed a set of $5$ polycubes: an encoder, a selector, a linker and two fillers.

\subsection{Proof of Theorem \ref{thm_3d_new}}

\begin{proof}[Proof of Theorem \ref{thm_3d_new}]
We prove by reduction from Wang's domino problem which is known to be undecidable by Theorem \ref{thm_berger}. For each instance of Wang's domino problem, we construct a set of $5$ polycubes as described in the previous subsection. It remains to show that the set of $5$ polycubes can tile the $3$-dimensional space if and only if the corresponding set of Wang tile can tile the plane. To this end, We will first show that to tile the space with the set of $5$ polycubes, we must follow the patterns in Figure \ref{fig_3d_pattern} and Figure \ref{fig_3d_pattern_2} for each horizontal layers (of level-2).


\begin{figure}[H]
\begin{center}
\begin{tikzpicture}[scale=0.5]

\foreach \x in {0,8}
\foreach \y in {0}
{
\draw [ fill=gray!20] (\x+0,0+\y)--(\x+0,8+\y)--(\x+8,8+\y)--(\x+8,0+\y)--(\x+0,0+\y);
\draw [ fill=orange!20] (\x+1,1+\y)--(\x+1,7+\y)--(\x+7,7+\y)--(\x+7,1+\y)--(\x+1,1+\y);
\node at (\x+0.5,\y+0.5) {$x$}; \node at (\x+7.5,\y+0.5) {$X$};
\node at (\x+4.5,\y+0.5) {$y$}; \node at (\x+0.5,\y+7.5) {$Y$};
}
\foreach \x in {-4,4,12}
\foreach \y in {-8}
{
\draw [ fill=gray!20] (\x+0,0+\y)--(\x+0,8+\y)--(\x+8,8+\y)--(\x+8,0+\y)--(\x+0,0+\y);
\draw [ fill=orange!20] (\x+1,1+\y)--(\x+1,7+\y)--(\x+7,7+\y)--(\x+7,1+\y)--(\x+1,1+\y);
\node at (\x+0.5,\y+0.5) {$x$}; \node at (\x+7.5,\y+0.5) {$X$};
\node at (\x+4.5,\y+0.5) {$y$}; \node at (\x+0.5,\y+7.5) {$Y$};
}
\foreach \x in {0,8}
\foreach \y in {-16}
{
\draw [ fill=gray!20] (\x+0,0+\y)--(\x+0,8+\y)--(\x+8,8+\y)--(\x+8,0+\y)--(\x+0,0+\y);
\draw [ fill=orange!20] (\x+1,1+\y)--(\x+1,7+\y)--(\x+7,7+\y)--(\x+7,1+\y)--(\x+1,1+\y);
\node at (\x+0.5,\y+0.5) {$x$}; \node at (\x+7.5,\y+0.5) {$X$};
\node at (\x+4.5,\y+0.5) {$y$}; \node at (\x+0.5,\y+7.5) {$Y$};
}

\end{tikzpicture}
\end{center}
\caption{The tiling pattern (first and second layers of selectors, ignore the labeled building blocks for the second layer).}\label{fig_3d_pattern}
\end{figure}
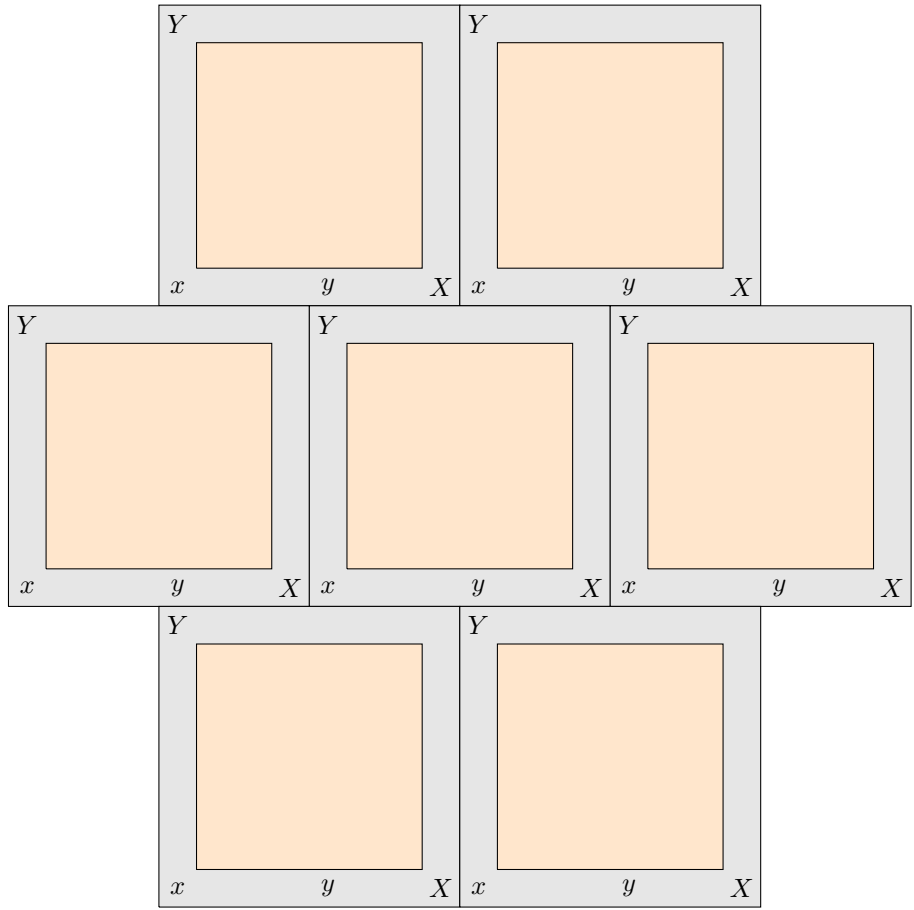

\begin{itemize}
    \item First of all, the selector must be used in any tiling of the set of $5$ polycubes. If the encoder is used, then the selector must be used, because the only building blocks that can match $A$ or $B$ are the building blocks $a$ or $b$ in the selector. If the linker or the fillers are used, then to match the bumps in the linker or the fillers, either the selector or the encoder must be used. In all cases, the selector must be used in order to tile the $3$-dimensional space.
    \item The selectors must form a $3$-dimensional lattice in any tilings. Because the building blocks $x$, $X$, $y$, $Y$, $z$, and $Z$ determine the tiling patterns of the selectors. As we have mentioned before, the selectors must form two-way infinite wells vertically in any tilings. Furthermore, the first layer and second layer of the selectors must form horizontal lattices as illustrated in Figure \ref{fig_3d_pattern}, and the third layer of the selectors must form horizontal lattices as illustrated in Figure \ref{fig_3d_pattern_2}. So the selectors are all aligned horizontally. In other words, the first layer (second layer, third layer, resp.) must be on the same altitude of the first layer (second layer or third layer, resp.) of another selector.
    \item Because of existence of the building blocks in the selectors, a two-way infinite array of vertical encoders must be placed inside each vertical well of selectors. Note that there cannot be a gap between two consecutive encoders in the vertical array, as there is nothing else that can fill such a gap. But there is a freedom to choose which of three simulated Wang tiles to be aligned with the third layer of the selectors, and the choice is independent for each vertical arrays of encoders. Therefore, we can also think of the tilings as the selectors select a simulated Wang tile independently for each vertical well to be exposed at the third layers.
    \item Finally, we need to fill up all the gaps left behind by the selectors and encoders. The small gaps (i.e. between the building blocks $c$ and $d$, $c^-$ and $d$, or $c$ and $d^-$ on the first or second layers) inside the selectors can be filled by either the standard filler $F$ or the bigger filler $F^+$. For the gaps between the encoders (through the selectors) on the third layer of the selector (see Figure \ref{fig_3d_pattern_2}), the only hope is to fill with linkers. By the length of the linker, it is easy to check the gaps between two simulated Wang tiles can be filled without gaps or overlaps by linkers if and only if the linkers are connecting a building block $c$ to a building block $d^-$, or $c^-$ to $d$. By the encoding method of the colors of Wang tiles, this is equivalent to that the two adjacent sides of the simulated Wang tiles are of the same color. Therefore, the third layers can be tiled without gaps or overlaps if and only if the simulated Wang tiles form a plane tiling.
    \item Note that the tiling in a third layer of the selectors is exactly the same as the tiling of any other third layers. Therefore, all the third layers simulate the same plane tiling with Wang tiles. If the set of $5$ polycubes tile the entire space, then the tiling is always periodic in the vertical direction.
\end{itemize}


\begin{figure}[H]
\begin{center}
\begin{tikzpicture}[scale=0.5]

\foreach \x in {0,8}
\foreach \y in {0}
{
\draw [ fill=gray!20] (\x+0,0+\y)--(\x+0,8+\y)--(\x+8,8+\y)--(\x+8,0+\y)--(\x+0,0+\y);
\draw [ fill=orange!20] (\x+1,1+\y)--(\x+1,7+\y)--(\x+7,7+\y)--(\x+7,1+\y)--(\x+1,1+\y);
\draw [ fill=violet!20] (\x+1,7+\y)--(\x+1,9+\y)--(\x+2,9+\y)--(\x+2,7+\y)--(\x+1,7+\y);
\draw [ fill=violet!20] (\x+3,7+\y)--(\x+3,9+\y)--(\x+2,9+\y)--(\x+2,7+\y)--(\x+3,7+\y);
\draw [ fill=violet!20] (\x+6,7+\y)--(\x+6,9+\y)--(\x+5,9+\y)--(\x+5,7+\y)--(\x+6,7+\y);
\draw [ fill=violet!20] (\x+6,7+\y)--(\x+6,9+\y)--(\x+7,9+\y)--(\x+7,7+\y)--(\x+6,7+\y);
}
\foreach \x in {-4,4,12}
\foreach \y in {-8}
{
\draw [ fill=gray!20] (\x+0,0+\y)--(\x+0,8+\y)--(\x+8,8+\y)--(\x+8,0+\y)--(\x+0,0+\y);
\draw [ fill=orange!20] (\x+1,1+\y)--(\x+1,7+\y)--(\x+7,7+\y)--(\x+7,1+\y)--(\x+1,1+\y);
\draw [ fill=violet!20] (\x+1,7+\y)--(\x+1,9+\y)--(\x+2,9+\y)--(\x+2,7+\y)--(\x+1,7+\y);
\draw [ fill=violet!20] (\x+3,7+\y)--(\x+3,9+\y)--(\x+2,9+\y)--(\x+2,7+\y)--(\x+3,7+\y);
\draw [ fill=violet!20] (\x+6,7+\y)--(\x+6,9+\y)--(\x+5,9+\y)--(\x+5,7+\y)--(\x+6,7+\y);
\draw [ fill=violet!20] (\x+6,7+\y)--(\x+6,9+\y)--(\x+7,9+\y)--(\x+7,7+\y)--(\x+6,7+\y);
}
\foreach \x in {0,8}
\foreach \y in {-16}
{
\draw [ fill=gray!20] (\x+0,0+\y)--(\x+0,8+\y)--(\x+8,8+\y)--(\x+8,0+\y)--(\x+0,0+\y);
\draw [ fill=orange!20] (\x+1,1+\y)--(\x+1,7+\y)--(\x+7,7+\y)--(\x+7,1+\y)--(\x+1,1+\y);
}
\foreach \x in {-4,4,12}
\foreach \y in {-16}
{
\draw [ fill=violet!20] (\x+1,7+\y)--(\x+1,9+\y)--(\x+2,9+\y)--(\x+2,7+\y)--(\x+1,7+\y);
\draw [ fill=violet!20] (\x+3,7+\y)--(\x+3,9+\y)--(\x+2,9+\y)--(\x+2,7+\y)--(\x+3,7+\y);
\draw [ fill=violet!20] (\x+6,7+\y)--(\x+6,9+\y)--(\x+5,9+\y)--(\x+5,7+\y)--(\x+6,7+\y);
\draw [ fill=violet!20] (\x+6,7+\y)--(\x+6,9+\y)--(\x+7,9+\y)--(\x+7,7+\y)--(\x+6,7+\y);
}
\foreach \x in {0,8}
\foreach \y in {-24}
{
\draw [ fill=violet!20] (\x+1,7+\y)--(\x+1,9+\y)--(\x+2,9+\y)--(\x+2,7+\y)--(\x+1,7+\y);
\draw [ fill=violet!20] (\x+3,7+\y)--(\x+3,9+\y)--(\x+2,9+\y)--(\x+2,7+\y)--(\x+3,7+\y);
\draw [ fill=violet!20] (\x+6,7+\y)--(\x+6,9+\y)--(\x+5,9+\y)--(\x+5,7+\y)--(\x+6,7+\y);
\draw [ fill=violet!20] (\x+6,7+\y)--(\x+6,9+\y)--(\x+7,9+\y)--(\x+7,7+\y)--(\x+6,7+\y);
}

\draw [thick] (17.5,-15)--(18.5,-14)--(19.5,-15)--(18.5,-16)--(17.5,-15);
\draw [thick] (21.5,-15)--(20.5,-14)--(19.5,-15)--(20.5,-16)--(21.5,-15);
\draw [thick] (16.5,-14)--(17.5,-13)--(18.5,-14)--(17.5,-15)--(16.5,-14);
\draw [thick] (20.5,-14)--(19.5,-13)--(18.5,-14)--(19.5,-15)--(20.5,-14);
\draw [thick] (20.5,-14)--(21.5,-13)--(22.5,-14)--(21.5,-15)--(20.5,-14);
\draw [thick] (17.5,-13)--(18.5,-14)--(19.5,-13)--(18.5,-12)--(17.5,-13);
\draw [thick] (21.5,-13)--(20.5,-14)--(19.5,-13)--(20.5,-12)--(21.5,-13);

\draw [color=black!20] (16.5,-14)--(22.5,-14);
\draw [color=black!20] (17.5,-15)--(21.5,-15);
\draw [color=black!20] (17.5,-13)--(21.5,-13);

\draw [color=black!20] (18.5,-16)--(18.5,-12);
\draw [color=black!20] (20.5,-16)--(20.5,-12);

\draw [color=black!20] (21.5,-15)--(21.5,-13);
\draw [color=black!20] (19.5,-15)--(19.5,-13);
\draw [color=black!20] (17.5,-15)--(17.5,-13);

\end{tikzpicture}
\end{center}
\caption{The tiling pattern II (thrid layer of selectors).}\label{fig_3d_pattern_2}
\end{figure}
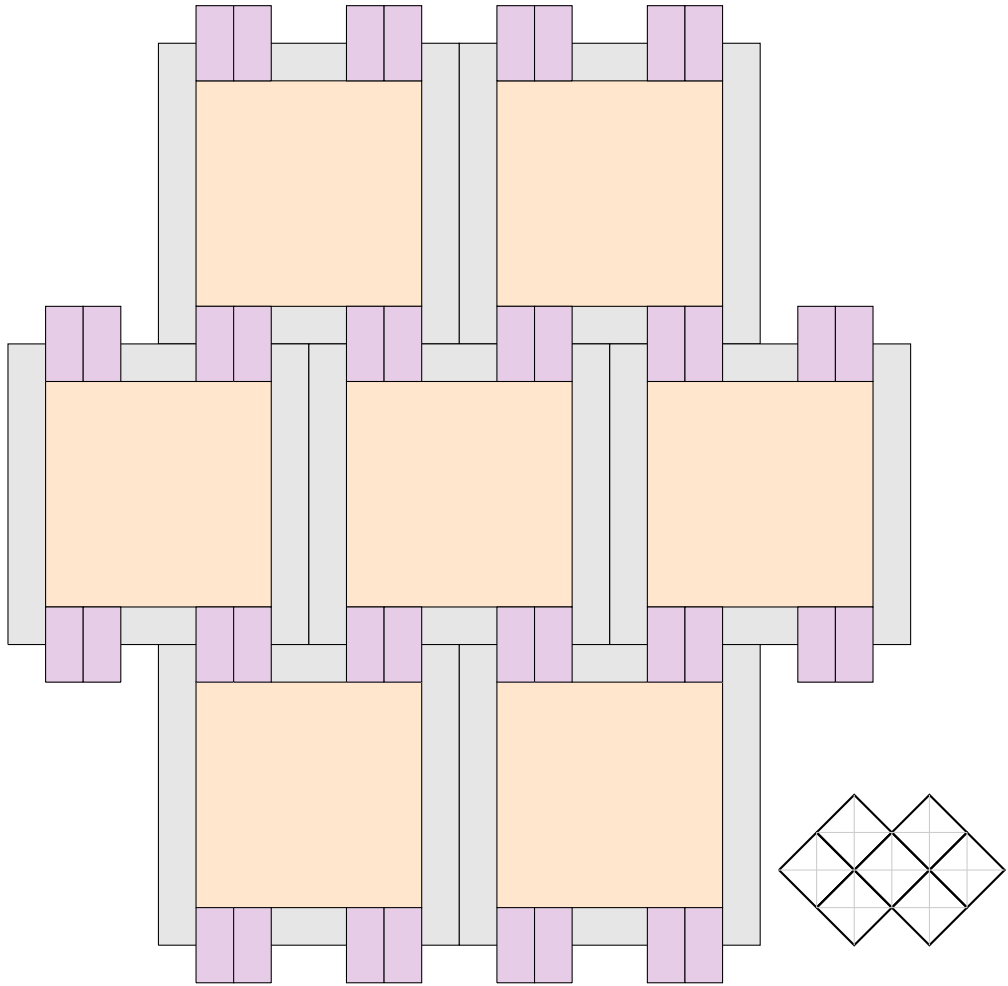

The above arguments have shown that the set of $5$ polycubes can tile the $3$-dimensional space if and only if the corresponding set of Wang tiles can tile the plane. This completes the proof.
\end{proof}

\noindent\textbf{Remark 1.} The reduction method in the proof of Theorem \ref{thm_3d_new} in many aspects is different with Ollinger's reduction framework introduced in \cite{o09}. In our method, the simulated Wang tiles are packed in a polycube (i.e. the encoder) vertically, and the simulated Wang tiles form tilings on the third layers of the selector horizontally. Thus, the dimensions of the space have been fully utilized to achieve a more compact arrangement of the simulated Wang tiles. This more compact pattern is the key to require only one linker polycube for connecting simulated Wang tiles, compared to at least two linkers in Ollinger's original framework.

\section{Undecidability of Tiling $4$-dimensional Space}\label{sec_4d}

We think of the $4$-dimensional Euclidean space as a $3$-dimensional space plus a fourth dimension of time. A $4$-dimensional polyhypercube\footnote{All hypercubes or polyhypercubes considered in this paper are $4$-dimensional.} can be imagined as polycubes evolve in finite steps of unit time frames. As a result, we describe a polyhypercube by consecutive frames of polycubes. 

\subsection{Building Blocks in $4$-dimensional Space}\label{ssec_4d_bb}

A \textit{polyhypercube} is a connected union of a finite number of $1\times 1\times 1\times 1$ unit hypercubes gluing together face-to-face. A $8\times 8\times 8\times 8$ polyhypercube is called a \textit{functional hypercube}. A $8\times 8\times 8\times 4$ polyhypercube is called a \textit{half functional hypercube}. A functional hypercube can be divided into two half functional hypercube: the former half and the latter half. A half functional hypercube consists of $4$ time frames, and each frame is a $3$-dimensional $8\times 8\times 8$ polycube. So a functional hypercube and a half functional hypercube are $3$-dimensional $8\times 8\times 8$ polycubes keeping unchanged over a time interval of length $8$ and $4$, respectively.

Let $K$ be a $3$-dimensional $8\times 8\times 8$ polycube, so $K$ is a set of $512$ unit cubes. Let $T_1$ denote the set of unit cubes on the outer surface of $K$. Let $T_2$ denote the set of unit cubes on the outer surface of $K-T_1$. Let $T_3$ denote the set of unit cubes on the outer surface of $K-T_1-T_2$. Let $T_4=K-T_1-T_2-T_3$. Thus, we have partitioned $K$ into disjoint union of $4$ sets: $T_1$, $T_2$, $T_3$ and $T_4$.

Now, we define the building block $c$ frame by frame. The building block $c$ is a half functional hypercube with a dent in the dimension of time, which consists of $4$ frames. The first frame of $c$ is the polycube $K$, the second frame is $T_1$, the third frame is $T_1\cup T_3$, and the fourth frame is $T_1\cup T_2\cup T_3$. So $c$ is a half functional hypercube with a dent facing the future.

The building block $C$ is a half functional hyper with with a bump facing the past which matches the dent of $c$ exactly. In other words, $C$ consists of $7$ frames (in order): $T_2\cup  T_3\cup T_4$, $T_2\cup T_4$, $T_4$, $K$, $K$, $K$, and $K$, where the first $3$ frames form the bump facing the past. The two building blocks $c$ and $C$ can be put together to form a functional hypercube perfectly without any gaps or overlaps. The building block $C$ alone forms a tile (the filler) in our set of $4$ tiles as we will see soon in Subsection \ref{ssec_4_tiles}.

Note that $T_4$ is a $2\times 2\times 2$ polycube at the very center of $K$. We divide $T_4$ into two parts, the upper half $T^+_4$ and the lower half $T^-_4$. Both $T^+_4$ and $T^-_4$ are $2\times 2\times 1$ polycubes. We can also divide $T_4$ into two halves in other ways. The south half $T^N_4$ and north half $T^S_4$, both of which are $2\times 1\times 2$ polycubes. The east half $T^E_4$ and the west half $E^W_4$, both of which are $1\times 2\times 2$ polycubes.

Both the building blocks $a$ and $b$ are half functional hypercubes with a dent facing the future. Let $a$ be the polyhypercube consist of $4$ frames: $K$, $T_1$, $T_1\cup T_3\cup T^-_4$, $T_1\cup T_2\cup T_3\cup T^-_4$. Let $b$ be the polyhypercube consist of $4$ frames: $K$, $T_1$, $T_1\cup T_3\cup T^+_4$, $T_1\cup T_2\cup T_3\cup T^+_4$. The building blocks $A$ and $B$ are half functional hypercubes with a bump facing the past that matches the dent of $a$ and $b$, respectively. In exact words, building block $A$ consists of $7$ frames: $T_2\cup  T_3\cup T_4$, $T_2\cup T^+_4$, $T^+_4$, $K$, $K$, $K$, and $K$; and $B$ consists of $7$ frames: $T_2\cup  T_3\cup T_4$, $T_2\cup T^-_4$, $T^-_4$, $K$, $K$, $K$, and $K$.

Three more pairs of building blocks, $x$ and $X$, $y$ and $Y$, and $z$ and $Z$, are defined similarly as follows:
\begin{itemize}
    \item $x$: $K$, $T_1$, $T_1\cup T_3\cup T^N_4$, $T_1\cup T_2\cup T_3\cup T^N_4$;
    \item $X$: $T_2\cup  T_3\cup T_4$, $T_2\cup T^S_4$, $T^S_4$, $K$, $K$, $K$, and $K$;
    \item $y$: $K$, $T_1$, $T_1\cup T_3\cup T^S_4$, $T_1\cup T_2\cup T_3\cup T^S_4$;
    \item $Y$: $T_2\cup  T_3\cup T_4$, $T_2\cup T^N_4$, $T^N_4$, $K$, $K$, $K$, and $K$;
    \item $z$: $K$, $T_1$, $T_1\cup T_3\cup T^W_4$, $T_1\cup T_2\cup T_3\cup T^W_4$;
    \item $Z$: $T_2\cup  T_3\cup T_4$, $T_2\cup T^E_4$, $T^E_4$, $K$, $K$, $K$, and $K$.
\end{itemize}

We need two more building blocks $\mathbb{E}$ and $\mathbb{S}$, which are $8\times 8\times 8\times 8$ functional hypercubes with both a dent and a bump. Recall that $T_4$ is a polycube consisting of $8$ unit cubes. Let $J$ be the unit cube at the top northeast corner. Both $\mathbb{E}$ and $\mathbb{S}$ consist of $11$ frames as defined below: 
\begin{itemize}
    \item $\mathbb{E}$: $T_2\cup  T_3\cup T_4$, $T_2\cup J$, $J$, $K$, $K$, $K$, $K$, $K$, $T_1$, $T_1\cup T_3\cup (T_4-J)$, $T_1\cup T_2\cup T_3\cup (T_4-J)$;
    \item $\mathbb{S}$: $T_2\cup  T_3\cup T_4$, $T_2\cup (T_4-J)$, $T_4-J$, $K$, $K$, $K$, $K$, $K$, $T_1$, $T_1\cup T_3\cup J$, $T_1\cup T_2\cup T_3\cup J$.
\end{itemize}

\subsection{Lifting the Tiles and Tilings}

In this subsection, we sketch the general idea of lifting tiles and tilings of $3$-dimensional space to tiles and tilings of $4$-dimensional space. This technique has been used in \cite{yz24b} to prove the undecidability of tiling the $3$-dimensional space with a set of $6$ polycubes, by lifting from $2$-dimensional tiling with a set of $8$ polyominoes. As it turns out in the subsequent subsections, this technique can be generalized to prove the main result of the current paper. The lifting technique contains the following interconnected ingredients.
\begin{itemize}
    \item Cut the $4$-dimensional spacetime into two-way infinite slices with respect to the fourth dimension (i.e. the time). Each slice is the $3$-dimensional space that evolves over a finite time interval of length $8$ units. In other words, a slice consists of $8$ frames, and each frame is the entire $3$-dimensional space. The slices are put one after another extending infinitely to the past and the future, therefore forming the $4$-dimensional spacetime. Each slice of spacetime is treated as a thick version of the $3$-dimensional space, and will be tiled almost the same way we tile the $3$-dimensional space in Section \ref{sec_3d}.
    \item The $4$-dimensional tile set are also lifted from the $3$-dimensional polycubes we introduced in Section \ref{sec_3d} accordingly. Besides thickening each polycube in the $4$-th dimension (i.e. the time dimension) when lifting, another important modification is the shape of the dents and bumps. All the dents and bumps are lifted to the $4$-th dimension, as we have seen in the $4$-dimensional building blocks in the previous subsection.
    \item The $4$-dimensional linker (which will be introduced in the next subsection) in our tile set can be misaligned (with respect to the time slice) when tiling the $4$-dimensional spacetime. In other words, it can be placed essentially across two adjacent slices. This is a very useful feature that plays an important role in decreasing the number of fillers from two in the $3$-dimensional tile set to just one in the $4$-dimensional tile set.
\end{itemize}

Note that the above is a high-level oversimplified description of the lifting technique. The details will be explained in the next two subsections. 

\subsection{The Set of $4$ Polyhypercubes}\label{ssec_4_tiles}

The $3$-dimensional projections of \textit{encoder}, \textit{selector}, \textit{linker} are illustrated by layer diagram in Figure \ref{fig_encoder_4d}, Figure \ref{fig_selector_4d} and Figure \ref{fig_linker_4d}, respectively. As these figures are $3$-dimensional projections of $4$-dimensional tiles, keep in mind that each square represents a $4$-dimensional building block. The gray squares without a label in these figures are normal $8\times 8\times 8\times 8$ functional hypercube. The gray squares with a label are building blocks we have introduced in Subsection~\ref{ssec_4d_bb}. Note that the main part (i.e. the functional hypercubes) consists of $8$ frames, but the other labeled building blocks may have more than or less than $8$ frames. So we have to specify  how the labeled building blocks are attached to the main part with respect to the time frames. A building block with a subscript $*$ is attached to the former half of the time frames (i.e. from the first frame to the fourth frame). A building block with a superscript $*$ is attached to the latter half of the time frames (i.e. from the fifth frame to the eighth frame). The building blocks $a$, $b$, $c$, $x$, $y$ and $z$ have exactly $4$ frames, so they are attached to either the first four frames or the last four frames of the main part. The building blocks $A$, $B$, $C$, $X$, $Y$, and $Z$ have $7$ frames, among which four frames form a half functional hypercube, and the rest three frames form a bump. We attach the half functional hypercube to the former half or the latter half of the main part. Similarly, the building blocks $\mathbb{E}$ and $\mathbb{S}$ consist of a functional hypercube (of $8$ frames, including the dent) and a bump (of $3$ frames). We attach the functional hypercube to the $8$ frames of the main part.


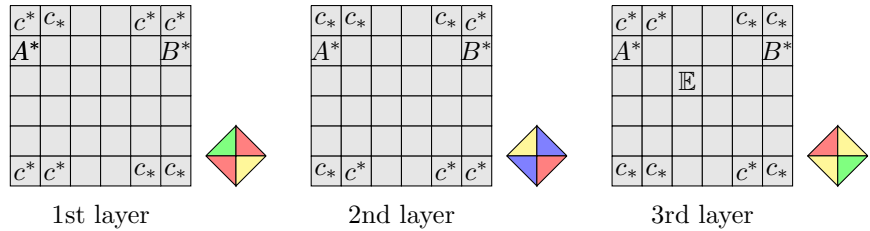
\begin{figure}[H]
\begin{center}
\begin{tikzpicture}[scale=0.4]

\foreach \x in {0,...,5}
\foreach \y in {0,...,5} 
{
\draw [ fill=gray!20] (\x+0,1+\y)--(\x+1,1+\y)--(\x+1,0+\y)--(\x+0,0+\y)--(\x+0,1+\y);
}

\foreach \x in {10,...,15}
\foreach \y in {0,...,5} 
{
\draw [ fill=gray!20] (\x+0,1+\y)--(\x+1,1+\y)--(\x+1,0+\y)--(\x+0,0+\y)--(\x+0,1+\y);
}

\foreach \x in {20,...,25}
\foreach \y in {0,...,5} 
{
\draw [ fill=gray!20] (\x+0,1+\y)--(\x+1,1+\y)--(\x+1,0+\y)--(\x+0,0+\y)--(\x+0,1+\y);
}

\foreach \x in {0,4,5,15,20,21}
\foreach \y in {5}  
{
\node at (\x+0.5,\y+0.5) {$c^*$};
}
\foreach \x in {1,10,11,14,24,25}
\foreach \y in {5}  
{
\node at (\x+0.5,\y+0.5) {$c_*$};
}

\foreach \x in {0,1,11,14,15,24}
\foreach \y in {0} 
{
\node at (\x+0.5,\y+0.5) {$c^*$};
}
\foreach \x in {4,5,10,20,21,25}
\foreach \y in {0} 
{
\node at (\x+0.5,\y+0.5) {$c_*$};
}

\foreach \x in {0,10,,20}
\foreach \y in {4} 
{
\node at (\x+0.5,\y+0.5) {$A^*$};
}
\foreach \x in {5,15,25}
\foreach \y in {4} 
{
\node at (\x+0.5,\y+0.5) {$B^*$};
}


\foreach \x in {22}
\foreach \y in {3} 
{
\node at (\x+0.5,\y+0.5) {$\mathbb{E}$};
}

\node at (3,-1) {$1$st layer};  \node at (13,-1) {$2$nd layer};  \node at (23,-1) {$3$rd layer};

\draw [fill=green!50] (7-0.5,1)--(8-0.5,1)--(8-0.5,2)--(7-0.5,1);
\draw [fill=red!50] (9-0.5,1)--(8-0.5,1)--(8-0.5,2)--(9-0.5,1);
\draw [fill=red!50] (7-0.5,1)--(8-0.5,1)--(8-0.5,0)--(7-0.5,1);
\draw [fill=yellow!50] (9-0.5,1)--(8-0.5,1)--(8-0.5,0)--(9-0.5,1);

\draw [fill=yellow!50] (17-0.5,1)--(18-0.5,1)--(18-0.5,2)--(17-0.5,1);
\draw [fill=blue!50] (19-0.5,1)--(18-0.5,1)--(18-0.5,2)--(19-0.5,1);
\draw [fill=blue!50] (17-0.5,1)--(18-0.5,1)--(18-0.5,0)--(17-0.5,1);
\draw [fill=red!50] (19-0.5,1)--(18-0.5,1)--(18-0.5,0)--(19-0.5,1);

\draw [fill=red!50] (27-0.5,1)--(28-0.5,1)--(28-0.5,2)--(27-0.5,1);
\draw [fill=yellow!50] (29-0.5,1)--(28-0.5,1)--(28-0.5,2)--(29-0.5,1);
\draw [fill=yellow!50] (27-0.5,1)--(28-0.5,1)--(28-0.5,0)--(27-0.5,1);
\draw [fill=green!50] (29-0.5,1)--(28-0.5,1)--(28-0.5,0)--(29-0.5,1);

\end{tikzpicture}
\end{center}
\caption{$3$-dimensional projection of the encoder.}\label{fig_encoder_4d}
\end{figure}


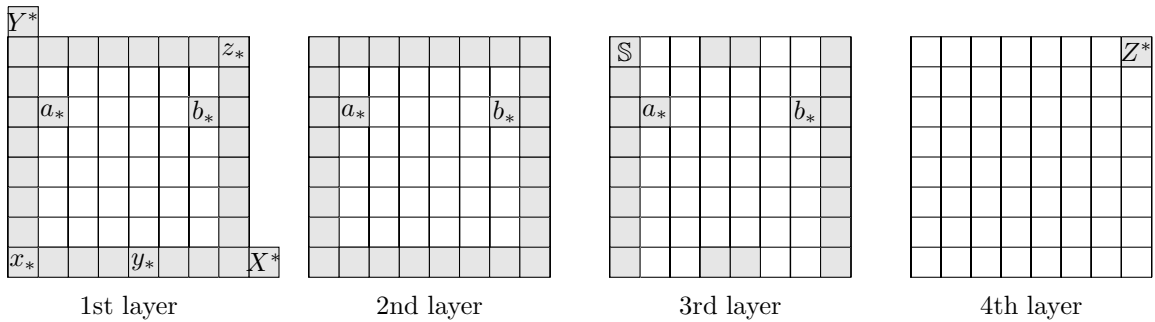
\begin{figure}[H]
\begin{center}
\begin{tikzpicture}[scale=0.4]

\foreach \x in {0,10,20,30}
\foreach \y in {0,...,8} 
{ 
\draw   (\x+0,0+\y)--(\x+8,0+\y);
\draw  (\x,0)--(\x,8);
\draw   (\x+8,0)--(\x+8,8);
\draw   (\x+7,0)--(\x+7,8);
\draw   (\x+6,0)--(\x+6,8);
\draw   (\x+5,0)--(\x+5,8);
\draw   (\x+4,0)--(\x+4,8);
\draw   (\x+3,0)--(\x+3,8);
\draw   (\x+2,0)--(\x+2,8);
\draw   (\x+1,0)--(\x+1,8);
}

\foreach \x in {0,...,7}
\foreach \y in {0,7} 
{
\draw [ fill=gray!20] (\x+0,1+\y)--(\x+1,1+\y)--(\x+1,0+\y)--(\x+0,0+\y)--(\x+0,1+\y);
}
\foreach \x in {8}
\foreach \y in {0} 
{
\draw [ fill=gray!20] (\x+0,1+\y)--(\x+1,1+\y)--(\x+1,0+\y)--(\x+0,0+\y)--(\x+0,1+\y);
}
\foreach \x in {0,7}
\foreach \y in {0,...,7} 
{
\draw [ fill=gray!20] (\x+0,1+\y)--(\x+1,1+\y)--(\x+1,0+\y)--(\x+0,0+\y)--(\x+0,1+\y);
}
\foreach \x in {0}
\foreach \y in {8} 
{
\draw [ fill=gray!20] (\x+0,1+\y)--(\x+1,1+\y)--(\x+1,0+\y)--(\x+0,0+\y)--(\x+0,1+\y);
}
\foreach \x in {1,6}
\foreach \y in {5} 
{
\draw [ fill=gray!20] (\x+0,1+\y)--(\x+1,1+\y)--(\x+1,0+\y)--(\x+0,0+\y)--(\x+0,1+\y);
}

\foreach \x in {10,17}
\foreach \y in {0,...,7} 
{
\draw [ fill=gray!20] (\x+0,1+\y)--(\x+1,1+\y)--(\x+1,0+\y)--(\x+0,0+\y)--(\x+0,1+\y);
}
\foreach \x in {11,...,16}
\foreach \y in {0,7} 
{
\draw [ fill=gray!20] (\x+0,1+\y)--(\x+1,1+\y)--(\x+1,0+\y)--(\x+0,0+\y)--(\x+0,1+\y);
}
\foreach \x in {11,16}
\foreach \y in {5} 
{
\draw [ fill=gray!20] (\x+0,1+\y)--(\x+1,1+\y)--(\x+1,0+\y)--(\x+0,0+\y)--(\x+0,1+\y);
}

\foreach \x in {20,27}
\foreach \y in {0,...,7} 
{
\draw [ fill=gray!20] (\x+0,1+\y)--(\x+1,1+\y)--(\x+1,0+\y)--(\x+0,0+\y)--(\x+0,1+\y);
}
\foreach \x in {23,24}
\foreach \y in {0,7} 
{
\draw [ fill=gray!20] (\x+0,1+\y)--(\x+1,1+\y)--(\x+1,0+\y)--(\x+0,0+\y)--(\x+0,1+\y);
}

\foreach \x in {21,26}
\foreach \y in {5} 
{
\draw [ fill=gray!20] (\x+0,1+\y)--(\x+1,1+\y)--(\x+1,0+\y)--(\x+0,0+\y)--(\x+0,1+\y);
}

\foreach \x in {37}
\foreach \y in {7} 
{
\draw [ fill=gray!20] (\x+0,1+\y)--(\x+1,1+\y)--(\x+1,0+\y)--(\x+0,0+\y)--(\x+0,1+\y);
}

\foreach \x in {8}
\foreach \y in {0} 
{
\node at (\x+0.5,\y+0.5) {$X^*$};
}
\foreach \x in {0}
\foreach \y in {0} 
{
\node at (\x+0.5,\y+0.5) {$x_*$};
}

\foreach \x in {0}
\foreach \y in {8} 
{
\node at (\x+0.5,\y+0.5) {$Y^*$};
}
\foreach \x in {4}
\foreach \y in {0} 
{
\node at (\x+0.5,\y+0.5) {$y_*$};
}

\foreach \x in {7}
\foreach \y in {7} 
{
\node at (\x+0.5,\y+0.5) {$z_*$};
}
\foreach \x in {37}
\foreach \y in {7} 
{
\node at (\x+0.5,\y+0.5) {$Z^*$};
}

\foreach \x in {1,11,21}
\foreach \y in {5} 
{
\node at (\x+0.5,\y+0.5) {$a_*$};
}
\foreach \x in {6,16,26}
\foreach \y in {5} 
{
\node at (\x+0.5,\y+0.5) {$b_*$};
}


\foreach \x in {20}
\foreach \y in {7} 
{
\node at (\x+0.5,\y+0.5) {$\mathbb{S}$};
}

\node at (4,-1) {$1$st layer};  \node at (14,-1) {$2$nd layer};  \node at (24,-1) {$3$rd layer};   \node at (34,-1) {$4$th layer}; 


\end{tikzpicture}
\end{center}
\caption{$3$-dimensional projection the selector.}\label{fig_selector_4d}
\end{figure}


\begin{figure}[H]
\begin{center}
\begin{tikzpicture}[scale=0.4]

\foreach \x in {0}
\foreach \y in {0,1,2,3} 
{
\draw [ fill=gray!20] (\x+0,1+\y)--(\x+1,1+\y)--(\x+1,0+\y)--(\x+0,0+\y)--(\x+0,1+\y);
}

\foreach \x in {0}
\foreach \y in {0,3} 
{
\node at (\x+0.5,\y+0.5) {$C^*$};
}


\end{tikzpicture}
\end{center}
\caption{$3$-dimensional projection of linker.}\label{fig_linker_4d}
\end{figure}
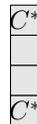

As we have mentioned in the previous subsection, the three tiles (encoder, selector, and filler) are lifted from their respective $3$-dimensional counterparts introduced in Section \ref{sec_3d}. They are thick (in the time dimension) versions of their $3$-dimensional counterparts with all the bumps and dents lifted to the fourth dimension. For the encoder (see Figure \ref{fig_encoder_4d}), the colors of simulated Wang tiles are now encoded by the way the building blocks are attached to the main part (i.e. attached to the former half or latter half). The colors red, green, blue and yellow are encoded by $c^*c^*$, $c^*c_*$, $c_*c^*$, and $c_*c_*$, respectively. For the linker (see Figure \ref{fig_linker_4d}), the two building blocks $C$ are both attached to the latter half of the main part. Therefore, the linkers can only connect two building blocks $c$ that are aligned with time.

Finally, a \textit{filler} is just a building block $C$. Thus we have constructed a set of $4$ polyhypercubes: an encoder, a selector, a linker, and a filler. It is straightforward to check they are all connected in $4$-dimensional space.

\subsection{Proof of Theorem \ref{thm_main}}

\begin{proof}[Proof of Theorem \ref{thm_main}]
We prove the theorem also by reducing Wang's domino problem. For each set $W$ of Wang tiles, we construct a set $P$ of $4$ polyhypercubes such that $W$ can tile the plane if and only if $P$ can tile the $4$-dimensional space. The $4$ polyhypercubes have already been described in the previous subsections. They are an encoder, a selector, a linker and a filler. We will show that to tile the $4$-dimensional spacetime with translated copies of these $4$ kinds of tiles, each slice must also follow the patterns in Figure \ref{fig_3d_pattern} and Figure \ref{fig_3d_pattern_2} (but think of them as $3$-dimesional projection of a slice of $4$-dimensional space).
\begin{itemize}
    \item With a similar argument as the $3$-dimensional case, the selector must be used in any tiling of the $4$-dimensional space. If the encoder is used, then the selector must be used as the only building blocks to match $A^*$ and $B^*$ are the building blocks $a_*$ and $b_*$ in the selector. If the linker or the filler are used, then the encoder must to used in order to match the building blocks $C$ in the linker or the filler. In all cases, the selector must be used.
    \item Within each slice, the spatial relation between the selectors is determined by the $3$ pair of building blocks $x_*$ and $X^*$, $y_*$ and $Y^*$, and $z_*$ and $Z^*$. The selectors form the same lattice lifted from the $3$-dimensional counterpart in Section \ref{sec_3d}. The building blocks $z_*$ and $Z^*$ determine that the selectors must form two-way infinite wells in the vertical direction. The building blocks $x_*$, $X^*$, $y_*$ and $Y^*$ determine that the selectors must form a lattice in every horizontal layer (see Figure \ref{fig_3d_pattern} and Figure \ref{fig_3d_pattern_2}). Note that the selectors are almost perfectly aligned within each slice ($8$ frames) except for the building block $\mathbb{S}$ with a bump interlocking with the previous slice (to be more precise, interlock with the last $3$ frames of the previous slice).
    \item The building block $\mathbb{S}$ determines that the overall structures of the selectors of the next slice and the previous slice must be exactly the same as the current slice.
    \item To match the building blocks $a_*$ and $b_*$, encoders must be placed inside every vertical well of selectors in every slice. Just like the $3$-dimensional case in the previous section, different wells of selectors in a slice can independently select one simulated Wang tile of the encoder to be aligned with the third layer of the selector. Note also that the building blocks $a_*$, $b_*$, $A^*$ and $B^*$ ensure that the encoders are aligned (in time) with the selectors within each slice except for the building blocks $\mathbb{E}$.
    \item The building block $\mathbb{E}$ determines that the overall structures of the encoders of the next slice and the previous slice must be exactly the same as the current slice. So the overall tiling structures of selectors and encoders are the same for every slice. Thus the tilings of $3$-dimensional space of Section \ref{sec_3d} have been lifted to $4$-dimensional space by repeating infinitely in the direction of time.
    \item Finally, there are some gaps between the selectors and encoders to be filled by the fillers or linkers. There are two kinds of gaps. The first kind is the smaller gaps between the building blocks $c^*$ or $c_*$ of the encoders and the first and second layers of selectors in every slice. These smaller isolated $4$-dimensional vacant holes can be filled exactly by the filler polyhypercubes (i.e. building blocks $C$). The second kind is the bigger gaps between the two adjacent encoders aligned with the third layer of selectors. For any pair of adjacent encoders (through the windows of the third layer of selectors), the gap between them is a two-way infinite time tunnel (i.e. the vacant space in one slice extends infinitely to the past and the future, see Figure \ref{fig_time_tunnel} for conceptual illustration). After being lifted from $3$-dimensional space, these otherwise finite isolated holes (in $3$-dimensional space) have also been lifted to infinite regions of $4$-dimensional space. These infinite time tunnels can be filled exactly if and only if each pair of encoding building blocks (of the two adjacent encoders) are either both attached to the former half of the slice, or both attached to the latter half of the slice. To connect two building blocks $c_*$, the linker is aligned with a slice. To connect two building blocks $c^*$, the linker is not aligned with the slices. This is equivalent to that any pair of adjacent edges of the simulated Wang tiles on the third layer have the same color.
\end{itemize}


\begin{figure}[H]
\begin{center}
\begin{tikzpicture}[scale=1]

\foreach \x in {0}
\foreach \y in {0} 
{
\draw [->,very thick] (\x+0,0+\y)--(\x+7,0+\y);
\draw [->,very thick] (\x+0,0+\y)--(\x+5,4+\y);
\draw [->,very thick] (\x+0,0+\y)--(\x+0,6+\y);

\draw (\x+4,\y+1)--(\x+5,\y+1)--(\x+5,\y+2)--(\x+3,\y+2)--(\x+3,\y+1)--(\x+4,\y+1)--(\x+4,\y+2)--(\x+6.5,\y+4)--(\x+7.5,\y+4)--(\x+7.5,\y+3)--(\x+5,\y+1);
\draw (\x+7.5,\y+4)--(\x+5,\y+2);
\draw (\x+6.5,\y+4)--(\x+5.5,\y+4)--(\x+3,\y+2);
\draw (\x+6.25,\y+2)--(\x+6.25,\y+3)--(\x+4.25,\y+3);
\draw (\x+6.875,\y+2.5)--(\x+6.875,\y+3.5)--(\x+4.875,\y+3.5);
\draw (\x+5.625,\y+1.5)--(\x+5.625,\y+2.5)--(\x+3.625,\y+2.5);
}

\node at (7.2,0) {$x$}; \node at (0,6.2) {$z$}; \node at (5.2,4.2) {$y$};
\node at (3.7,2.2) {$C^*$}; \node at (4.7,2.2) {$C_*$};
\node at (5.7,3.7) {$C^*$}; \node at (6.7,3.7) {$C_*$};
 
\foreach \x in {8}
\foreach \y in {4} 
{
\draw [->,very thick] (\x+0,0+\y)--(\x+7,0+\y);
\draw [->,very thick] (\x+0,0+\y)--(\x+5,4+\y);
\draw [->,very thick] (\x+0,0+\y)--(\x+0,6+\y);

\draw (\x+4,\y+1)--(\x+5,\y+1)--(\x+5,\y+2)--(\x+3,\y+2)--(\x+3,\y+1)--(\x+4,\y+1)--(\x+4,\y+2)--(\x+6.5,\y+4)--(\x+7.5,\y+4)--(\x+7.5,\y+3)--(\x+5,\y+1);
\draw (\x+7.5,\y+4)--(\x+5,\y+2);
\draw (\x+6.5,\y+4)--(\x+5.5,\y+4)--(\x+3,\y+2);
\draw (\x+6.25,\y+2)--(\x+6.25,\y+3)--(\x+4.25,\y+3);
\draw (\x+6.875,\y+2.5)--(\x+6.875,\y+3.5)--(\x+4.875,\y+3.5);
\draw (\x+5.625,\y+1.5)--(\x+5.625,\y+2.5)--(\x+3.625,\y+2.5);

\node at (\x+3.7,\y+2.2) {$C^*$}; \node at (\x+4.7,\y+2.2) {$C_*$};
\node at (\x+5.7,\y+3.7) {$C^*$}; \node at (\x+6.7,\y+3.7) {$C_*$};
}

\node at (15.2,4) {$x$}; \node at (8,10.2) {$z$}; \node at (13.2,8.2) {$y$};

\draw [color=red, very thick,->] (1,-1)--(5,1)--(13,5)--(15,6);
\node at (15.2,6.2) {$t$};


\end{tikzpicture}
\end{center}
\caption{A time tunnel.}\label{fig_time_tunnel}
\end{figure}
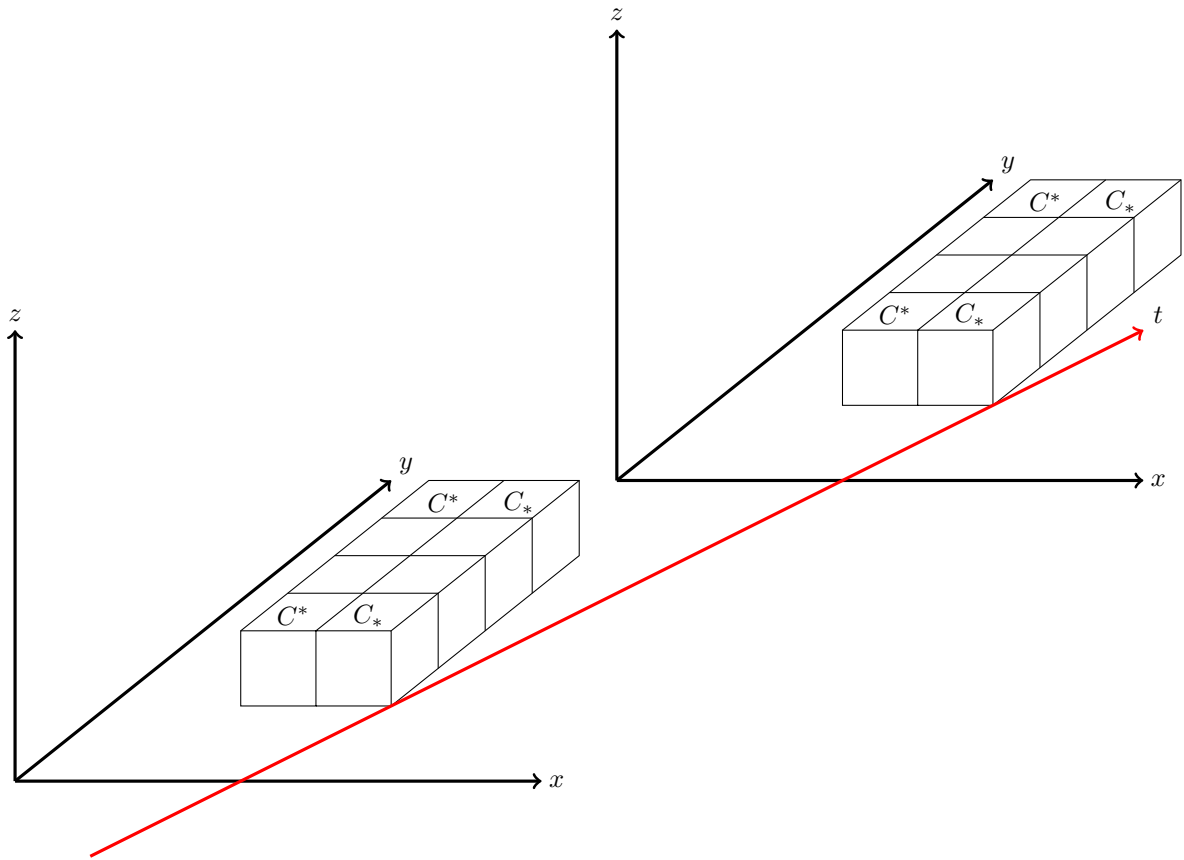

The tiling of the time tunnel with linkers in the last step in the above argument is the most crucial one. It means that the $4$ polyhypercubes can tile the $4$-dimensional space if and only if the set of Wang tiles being simulated can tile the plane. This completes the proof. 
\end{proof}

\section{Conclusion}\label{sec_conclu}
We have shown that the translational tiling problem of $\mathbb{Z}^n$ with a set of $k$ tiles is undecidable for $(k,n)=(4,4)$. The next step is to show the undecidability of tiling $\mathbb{Z}^n$ with a set of $k (\leq 3)$ tiles for some fixed dimension $n$. Ultimately, is it undecidable to tile $\mathbb{Z}^n$ with a single tile for some fixed $n$?

Both undecidability results (Theorem \ref{thm_main} and Theorem \ref{thm_3d_new}) of this paper are obtained for connected tiles (polycubes or polyhypercubes). For the more general case where tiles can be disconnected, is it possible to get the undecidability result with fewer tiles for translational tiling of $\mathbb{Z}^n$ with fixed $n$?

\section*{Acknowledgements}
The authors would like to thank Tal Kagalovsky for pointing out a flaw in the construction of the bigger filler $F^+$ in the first version of this paper. The first author was supported by the Research Fund of Guangdong University of Foreign Studies (Nos. 297-ZW200011 and 297-ZW230018), and the National Natural Science Foundation of China (No. 61976104).



\end{document}